\documentclass{article}
\usepackage{amsmath,amsfonts,euscript,amscd,amsthm,amssymb,upref,graphics,color,verbatim, float, placeins,mathrsfs, mathtools,comment,tikz}
\usepackage[all]{xy}
\usepackage{hyperref}
\usepackage[shortlabels]{enumitem}
\DeclarePairedDelimiter\floor{\lfloor}{\rfloor}

{\begin{enumerate}\setlength{\itemsep}{#1}}{\end{enumerate}}

\newcommand{\lnd}{\operatorname{{\rm LND}}}
\newcommand{\hlnd}{\operatorname{{\rm HLND}}}
\newcommand{\Integ}{\ensuremath{\mathbb{Z}}}
\newcommand{\Nat}{\ensuremath{\mathbb{N}}}
\newcommand{\Rat}{\ensuremath{\mathbb{Q}}}
\newcommand{\Comp}{\ensuremath{\mathbb{C}}}
\newcommand{\Reals}{\ensuremath{\mathbb{R}}}
\newcommand{\aff}{\ensuremath{\mathbb{A}}}
\newcommand{\bk}{{\ensuremath{\rm \bf k}}}
\newcommand{\lb}{\langle}
\newcommand{\rb}{\rangle}
\newcommand{\trdeg}{	\operatorname{{\rm trdeg}}}
\newcommand{\Frac}{		\operatorname{{\rm Frac}}}

\newcommand{\pr}{		\operatorname{{\rm pr}}}
\newcommand{\lcm}{		\operatorname{{\rm lcm}}}
\newcommand{\Pic}{		\operatorname{{\rm Pic}}}

\newcommand{\cotype}{		\operatorname{{\rm cotype}}}
\newcommand{\setspec}[2]{\big\{\,#1\, : \,#2\, \big\}}
\newcommand{\isom}{\cong}
\newcommand{\Sing}{		\operatorname{{\rm Sing}}}
\newcommand{\PPP}{\mathbb{P}}
\newcommand{\height}{		\operatorname{{\rm ht}}}

\newcommand{\Spec}{		\operatorname{{\rm Spec}}}
\newcommand{\Proj}{		\operatorname{{\rm Proj}}}
\newcommand{\NS}{		\operatorname{{\rm NS}}}

\newcommand{\Vol}{\operatorname{{\rm Vol}}}
\newcommand{\Supp}{		\operatorname{{\rm Supp}}}

\newcommand{\OSheaf}{\operatorname{\mathcal O}}

\newtheorem{theorem}[subsubsection]{Theorem}
\newtheorem*{theorem*}{Theorem}
\newtheorem{proposition}[subsubsection]{Proposition}
\newtheorem*{proposition*}{Proposition}
\newtheorem{lemma}[subsubsection]{Lemma}
\newtheorem{corollary}[subsubsection]{Corollary}

\theoremstyle{definition}

\newtheorem{remark}[subsubsection]{Remark}
\newtheorem{definition}[subsubsection]{Definition}
\newtheorem{definitions}[subsubsection]{Definitions}
\newtheorem{nothing}[subsubsection]{}
\newtheorem{nothing*}[subsubsection]{}
\newtheorem{example}[subsubsection]{Example}

\newtheorem{assumption}[subsubsection]{Assumption}
\newtheorem{notation}[subsubsection]{Notation}

\newtheorem{caution}[subsubsection]{Caution}
\newtheorem*{mainconjecture}{Main Conjecture}
\newtheorem*{maintheorem}{Main Theorem}
\newtheorem*{reductiontheorem}{Reduction Theorem}

\newtheorem*{maincor}{Corollary}
\newcommand{\codim}{		\operatorname{{\rm codim}}}

\newcommand{\pgoth}{\mathfrak{p}}
\newcommand{\qgoth}{\mathfrak{q}}

\newcommand{\Div}{		\operatorname{{\rm Div}}}

\renewcommand{\div}{	\operatorname{{\rm div}}}
\newcommand{\Cl}{		\operatorname{{\rm Cl}}}
\newcommand{\CaDiv}{\operatorname{{\rm CaDiv}}}
\newcommand{\CaCl}{		\operatorname{{\rm CaCl}}}
\newcommand{\mult}{\operatorname{{\rm mult}}}

\newcommand{\Meul}{\EuScript{M}}
\newcommand{\Aff}{\mathbb{A}}

\title{The Rigid Pham-Brieskorn Threefolds}
\author{Michael Chitayat and Adrien Dubouloz}
\date{January 2025}

\begin{document}

\maketitle

\begin{abstract}
We show that a $3$-dimensional Pham-Brieskorn hypersurface $\{ X_0^{a_0} + X_1^{a_1} + X_2^{a_2} + X_3^{a_3}=0\}$ in $\mathbb{A}^4$ such that $\min\{a_0, a_1, a_2, a_3 \} \geq 2$ and at most one element $i$ of $\{0,1,2,3\}$ satisfies $a_i = 2$ does not admit a non-trivial action of the additive group $\mathbb{G}_a$. 
\end{abstract}

\section*{Introduction}

An affine variety  $X$ defined over a field $\bk$ of characteristic zero is called \textit{rigid} if it does not admit a non-trivial action of the additive group $\mathbb{G}_{a,\bk}$. Equivalently, if the coordinate ring $B=\Gamma(X,\mathcal{O}_X)$ does not admit a non-zero locally nilpotent $\bk$-derivation, then $B$ is called a \textit{rigid ring}. In this article, we  study the rigidity of a class of varieties called \textit{affine Pham-Brieskorn hypersurfaces}. These varieties, denoted $X_{a_0,\ldots, a_n}=\mathrm{Spec}(B_{a_0,\ldots,a_n})$ are defined by equations of the form $X_0^{a_0} + \cdots +  X_n^{a_n}=0$ in $\mathbb{A}^{n+1}_{\bk}$, where $n\geq 2$ and $a_0,\dots,a_n$ are positive integers. If $a_i=1$ for some $i\in\{0,\ldots,n\}$ then $X_{a_0,\ldots, a_n}$ is isomorphic to $\mathbb{A}^n_{\bk}$ which is clearly not rigid.  If $\bk$ contains $\mathbb{Q}(i)$ and two of the $a_i$ (say $a_0$ and $a_1$) equal $2$, then $X_{a_0,\ldots, a_n}$ is isomorphic to the hypersurface $uv + X_2^{a_2}+\cdots +X_n^{a_n} = 0$ in $\mathbb{A}^{n+1}_{\bk}$. This hypersurface admits non-trivial $\mathbb{G}_{a,\bk}$-actions associated, for instance, to every locally nilpotent $\bk[u]$-derivation $\partial_i=u\tfrac{\partial}{\partial X_i}-a_i{X_i}^{a_i-1}\tfrac{\partial}{\partial v}$, $i=2,\ldots,n$, of its coordinate ring\footnote{Note that the real Pham-Brieskorn surface $X_{2,2,2}=\mathrm{Spec}(\Reals[X,Y,Z] / \lb X^2 + Y^2 + Z^2 \rb)$ is rigid (see e.g. \cite[Theorem 9.27]{freudenburg2017algebraic}), so the condition that $\bk$ contains $\mathbb{Q}(i)$ cannot be relaxed.}. In light of these observations, the following conjecture was proposed in \cite{flenner2003rational,Kali-Zaid_2000}: 
    
\begin{mainconjecture}\label{PBConjecture}
    An affine Pham-Brieskorn hypersurface  $X_{a_0, \dots, a_n}$ over an algebraically closed field of characteristic zero is rigid if and only if  $\min\{a_0, \dots,a_n \} \geq 2$ and at most one element $i$ of $\{0,\dots ,n\}$ satisfies $a_i = 2$. 
\end{mainconjecture}

The conjecture was proved by Kaliman and Zaidenberg \cite[Lemma 4]{Kali-Zaid_2000} in the $n=2$ case; despite several known partial results during the last decade (see in particular \cite{freudenburg2013}, \cite{DFM2017}, \cite{affineCones}, \cite{cheltsov_park_won_2016}, \cite{CylindersInDelPezzoSurfaces}, \cite{Chitayat_Daigle_2019}, \cite{LiuSunExtensions}) the conjecture remained open in dimension $n\geq 3$ (see \cite[Conjecture 1.22]{cheltsov2020cylinders}). In \cite{ChitayatThesis}, the first author nearly completed the proof of the $n=3$ case of the conjecture with the $(2,3,4,12)$ and $(2,3,5,30)$ cases remaining unsolved. 

In this article, we prove the $(2,3,4,12)$ and $(2,3,5,30)$ cases, completing the $n=3$ case of the conjecture. Our exposition also simplifies many of the results in \cite{ChitayatThesis} and reduces the problem (in all dimensions) to the cases where $\Proj B_{a_0, \dots, a_n}$ is a well-formed Fano variety. In order to provide the reader with a self-contained proof of the $n=3$ case of the Main Conjecture, we include all the necessary statements from $\cite{ChitayatThesis}$ as well as the proofs of those results that demonstrate the main ideas.  

Every affine Pham-Brieskorn hypersurface $X=X_{a_0,\ldots, a_n}$ identifies with the affine cone over the weighted hypersurface $\hat{X}$ defined by the weighted homogeneous equation $X_0^{a_0} + \cdots +  X_n^{a_n}=0$ in the weighted projective space $\mathbb{P}=\mathbb{P}(w_0,\ldots, w_n)$, where $w_i= \lcm(a_0, \dots, a_n)/a_i$. We say that  $X$ is a \textit{well-formed affine Pham-Brieskorn hypersurface} if the corresponding $\hat{X}$ is a well-formed weighted hypersurface in $\mathbb{P}$; that is, if  $\gcd(w_0,\dots,w_{i-1},\hat{w_i},w_{i+1}, \dots , w_n) = 1$ for every  $i = 0,\dots,n$ and $\codim_{\hat{X}}(\hat{X} \cap \Sing(\mathbb{P})) \geq 2$. Proposition \ref{PBWellFormed} provides a convenient arithmetic characterization of these well-formed hypersurfaces: they are exactly those for which $a_i$ divides $\lcm(a_0,\ldots, \hat{a}_i,\ldots a_n)$ for every $i=0,\ldots, n$. 

The following is a reduction of the Main Conjecture to the special case of well-formed affine Pham-Brieskorn hypersurfaces and an induction on the dimension:

\begin{reductiontheorem}\cite[Theorem 1.3.12]{ChitayatThesis} Let $n \geq 3$. The Main Conjecture holds in dimension $n$  provided it holds both in dimension $n-1$ and for well-formed affine Pham-Brieskorn hypersurfaces $X_{a_0,\ldots, a_n}$ of dimension $n$.
\end{reductiontheorem}

In particular, since the Main Conjecture holds for $n=2$, its verification in the case $n=3$ is reduced to the case of well-formed affine Pham-Brieskorn threefolds. Building on this result, we complete the proof of the Main Conjecture in the $n=3$ case, namely:

\begin{maintheorem}
   Let $X_{a_0,a_1,a_2,a_3}=\mathrm{Spec}(B_{a_0,a_1,a_2,a_3})$ be an affine Pham-Brieskorn threefold hypersurface.  If $\min\{a_0, a_1, a_2,a_3 \} \geq 2$ and at most one element $i$ of $\{0,1,2,3\}$ satisfies $a_i = 2$, then $X$ is rigid.
\end{maintheorem}

\begin{maincor} The Main Conjecture is true for $n=4$ if and only if it is true for well-formed affine Pham-Brieskorn fourfolds $X_{a_0,\ldots, a_4}$.
\end{maincor}

Let us briefly explain the scheme of the proof of the Main Theorem. Reducing to well-formed affine Pham-Brieskorn hypersurfaces $X=X_{a_0,a_1,a_2, a_3}$ with associated weighted hypersurfaces $\hat{X}$ in $\mathbb{P}(w_0,w_1,w_2, w_3)$ leads us to consider the following two sub-classes: 
\begin{enumerate}[\rm(1)]
 \item those $\hat{X}$ with pseudoeffective canonical divisor $K_{\hat{X}}$;
 \item a finite list of cases for which $\hat{X}$ is a del Pezzo surface with cyclic quotient singularities.
\end{enumerate}

The rigidity of well-formed affine Pham-Brieskorn hypersurfaces $X$ (of any dimension greater or equal to 2) for which $K_{\hat{X}}$ is pseudoeffective follows from Proposition \ref{noCylinder} which shows that a normal projective variety with pseudoeffective $\Rat$-Cartier canonical divisor and log-canonical singularities cannot contain a cylinder. In the second case, which contains for instance the affine cone $X_{3,3,3,3}$ over the Fermat cubic surface in $\mathbb{P}^3$,  it is known (by a general principle established by Kishimoto, Prokhorov and Zaidenberg in \cite{KishimotoProkhorovZaidenberg}) that the rigidity of $X$ is equivalent to the non-existence of an anti-canonical polar cylinder in the del Pezzo surface $\hat{X}$ (see Subsection \ref{subsec:cylinders} for the definition). Several cases in (2) turn out to be covered by a series of results on the classification of anti-canonical polar cylinders in del Pezzo surfaces with Du Val singularities due to Cheltsov, Park and Won \cite{cheltsov_park_won_2016}. We are eventually left with the study of the two cases $(a_0,a_1,a_2,a_3) \in \{(2,3,4,12) , (2,3,5,30)\}$ for which the corresponding del Pezzo surfaces have cyclic quotient singularities that are not Du Val. In each of these two cases, we are able to extract from the specific geometry of the del Pezzo surface at hand the non-existence of anti-canonical polar cylinders. These two particular examples raise and motivate the general problem of finding methods to expand the classification given in \cite{cheltsov_park_won_2016} to include del Pezzo surfaces with klt singularities.

\medskip

The article is organized as follows. Section \ref{Sec:GeoPrelims} gathers the preliminary algebraic and geometric  results required in Section \ref{Sec:Cotype0}. In particular, we include facts about locally nilpotent derivations, $\Rat$-divisors, weighted hypersurfaces in weighted projective spaces, cylinders and polar cylinders. We also include some necessary results about singularities of log pairs and the birational geometry of singular del Pezzo surfaces. In Section \ref{sec:reduction}, we establish the theorem reducing the Main Conjecture to the case of well-formed hypersurfaces and review the known relationship between the rigidity of the ring $B_{a_0,\dots, a_n}$ and the non-existence of canonical and anti-canonical polar cylinders in $\Proj(B_{a_0, \dots, a_n})$. We also prove rigidity of $B_{a_0,\dots, a_n}$ when $\Proj B_{a_0, \dots, a_n}$ has trivial or ample canonical divisor. Section \ref{Sec:Cotype0} resolves the $n = 3$ case of the Main Conjecture. Finally, Section \ref{Sec:Higher} gives some remarks on the Main Conjecture in higher dimensions. 
\\
\\
\textbf{Acknowledgements.} This project was partially funded by the Mitacs Globalink Research Award number IT33113 which helped  enable this collaboration. The first author would like to thank l'Institut de Math\'ematiques de Bourgogne for hosting and partially funding his visit to Dijon, and the University of Ottawa for facilitating his visit. The second author received partial support from the French ANR Project ANR-18-CE40-0003-01.

\section*{Assumptions, Conventions and Notation}

We assume the following throughout this article. 
\begin{itemize}
    \item All rings are commutative, associative, unital and have characteristic zero.
    \item We use the symbols $\Nat, \Nat^+, \Integ$ to denote the sets of natural numbers, positive integers and integers respectively. 
    \item We use the notation $``\subseteq"$ and $``\supseteq"$ for inclusion and containment and  $``\subset"$ and $``\supset"$ for proper inclusion and proper containment.
    \item Given an $\Nat$-graded ring $B = \bigoplus_{i \in \Nat} B_i$, we denote the irrelevant ideal by $B_+$. 
    \item All varieties and schemes are assumed to be defined over an algebraically closed field $\bk$ of characteristic zero which we fix throughout, a \textit{variety} being by convention an integral separated scheme of finite type. A \textit{curve} is a one-dimensional variety. A \textit{surface} is a two-dimensional variety. In particular, curves and surfaces are irreducible and reduced.  

    \item Whenever we discuss \textit{divisors} on a variety $X$, we assume implicitly that $X$ is normal. \item The function field of a variety $X$ is denoted by $K(X)$. The groups $\Div(X)$ and $\CaDiv(X)$ denote the groups of \textit{Weil divisors} and \textit{Cartier divisors} of $X$. The groups $\Cl(X)$, $\CaCl(X)$ and $\Pic(X)$ respectively denote the \textit{divisor class} and \textit{Cartier class} and \textit{Picard} groups of $X$. Since $X$ is integral, $\CaCl(X)$ is canonically isomorphic to the Picard group. Recall also that if $X$ is regular, the natural map $\CaCl(X)\to \Cl(X)$ is an isomorphism.    
\end{itemize}

\section{Preliminaries}\label{Sec:GeoPrelims}
\label{Sec:algPrelims}

\subsection{Locally Nilpotent Derivations}

\begin{nothing}
Let $B$ be a ring. A derivation $D:B \to B$ is \textit{locally nilpotent} if for every $b \in B$ there exists some $n \in \Nat$ such that $D^n(b) = 0$. If $B$ is a ring, the set of locally nilpotent derivations $D : B \to B$ is denoted $\lnd(B)$. If $\lnd(B) = \{0\}$, we say that $B$ is \textit{rigid}. Given $D \in \lnd(B)$, an element $t \in B$ is called a \textit{local slice of $D$} if $D(t) \neq 0$ and $D^2(t) = 0$. In particular, if $D(t) = 1$, then $t$ is a \textit{slice of $D$}. A derivation $D : B \to B$ is \textit{reducible} if there exists some $b \in B$ such that $D(B) \subseteq \lb b \rb \neq B$. If no such $b$ exists, then $D$ is \textit{irreducible}. 

Let $A \subseteq B$ be an inclusion of integral domains. Then, $A$ is \textit{factorially closed in $B$} if for all $x,y \in B \setminus \{0\}$, $xy \in A$ implies that $x,y \in A$. 

We say that a ring $B$ is \textit{$\bk$-affine} if it is finitely generated as a $\bk$-algebra. We call $B$ a \textit{$\bk$-domain} if it is a $\bk$-algebra that is also an integral domain. 
\end{nothing}

\begin{definitions}\cite[Section 2.3]{freudenburg2013}\label{absoluteDegree}
    Let $B$ be a $\bk$-domain. Given $b \in B \setminus \{0\}$ and $D \in \lnd(B)$,  define $\deg_D(b) = \max\{n \in \mathbb{N} : D^n(b) \neq 0\}$; define $\deg_D(0) = -\infty$. It is well-known that the map $\deg_D : B \to \Nat \cup \{-\infty\}$ is a degree function.

    If $B$ is not rigid, then given $f \in B$, the	\textit{absolute degree} of $f$ is defined by
		$$ |f|_B = \min\setspec{\deg_D(f)}{D \in \lnd(B) \setminus\{0\}}.$$
		If $B$ is rigid, define $|f|_B = -\infty$ if $f = 0$, and $|f|_B = \infty$ otherwise. 
\end{definitions}

\begin{theorem} \label {p0cfi2k309cbqp90ws}
		Let $B$ be an integral domain, let $D : B \to B$ be a derivation, and let $A = \ker(D)$. The following facts are well known. (Refer to  \cite{freudenburg2017algebraic} for instance.) 
		
		\begin{enumerate}	
			\item[\rm(a)] \label{c0ovjn3vr7} If $D$ is locally nilpotent, then $A$ is a factorially closed subring of $B$.
			Consequently, if $D$ is locally nilpotent and $\bk$ is a field included in $B$, then $D$ is a $\bk$-derivation.

			\item[\rm(b)] \label{teik5i68a9we} Assume that $\Rat \subseteq B$.
			If $D \neq 0$ is locally nilpotent then $D$ has a local slice $t \in B$.
			For any such $t$, if we define $s = D(t)$ then $B_s = A_s[t]$ is a polynomial ring in one variable over $A_s$.
			Consequently, $\trdeg(B/A) = 1$ and $\Frac(B) \supset \Frac(A)$ is purely transcendental of transcendence degree one.
		
			\item[\rm(c)] If $b \in B \setminus \{0\}$, then the derivation $bD : B \to B$ is locally nilpotent if and only if $D$ is locally nilpotent and $b \in A$.
			
			\item[\rm(d)] \label {02dj7edj9w34diey} If $D \neq 0$ is locally nilpotent and $B$ satisfies the ascending chain condition for principal ideals, then there exists an irreducible	locally nilpotent derivation $\delta : B \to B$ such that $D = a \delta$ for some $a \in A$.  
		\end{enumerate}
	\end{theorem}

	\begin{nothing}  \label {0j3w4c65mdry7u55gfyh6os}
        Let $B = \bigoplus_{n \in \Integ} B_n$ be a $\Integ$-graded ring. A derivation $D : B \to B$ is \textit{homogeneous} if there exists an $h\in \Integ$ such that $D(B_g) \subseteq B_{g+h}$ for all $g \in \Integ$. If $D$ is homogeneous and $D \neq 0$, then $h$ is unique and we say that $D$ is \textit{homogeneous of degree $h$}. The zero derivation is said to be homogeneous of degree $-\infty$. The set of homogeneous locally nilpotent derivations of $B$ is denoted $\hlnd(B)$.

        A \textit{graded subring} of $B$ is a subring $A$ of $B$ satisfying  $A = \bigoplus_{n \in \Integ} (A \cap B_n)$.	If $A$ is a graded subring of $B$ then $A$ too is a $\Integ$-graded ring. In particular, if $D \in \hlnd(B)$ then $\ker(D)$ is a $\Integ$-graded subring of $B$.        
	\end{nothing}

        \begin{proposition}(\cite[p.57]{Dai:TameWild}\label{homogenization})
            If $B$ is a $\Integ$-graded affine $\bk$-domain, then $\lnd(B) \neq \{0\}$ if and only if $\hlnd(B) \neq \{0\}$. 
        \end{proposition}
        
\subsection{The cotype of a tuple}

	\begin{definitions}(\cite[Section 3.9]{Chitayat_Daigle_2019}\label{typeDef})
		Let $n \geq 2$ and $S = (b_0, \dots,  b_n) \in \Integ^{n+1}$.  
		\begin{itemize}
			
			\item Define\footnote{By convention, $\gcd(S)\geq 0$ and $\lcm(S)\geq 0$.}  $\gcd(S) = \gcd(b_0, \dots, b_n)$ and $\lcm(S) = \lcm(b_0,  \dots,  b_n)$.
			
			\item If $\gcd(S) = 1$, we say that $S$ is \textit{normal}.
			If $S \neq (0, \dots, 0)$ and $d = \gcd(S)$, then the tuple $S' = (\frac{b_0}{d},  \dots,   \frac{b_n}{d})$ is normal,
			and is called the \textit{normalization of $S$}.
			
			\item For each $j \in \{0,\dots,n\}$, define $S_j = (b_0, \dots , \widehat{b_j} , \dots, b_n)$.
			
			\item We define 			$\cotype(S) = |\setspec{i \in \{ 0, \dots, n \} }{ \lcm(S_i) \neq \lcm(S) }| =| \setspec{i \in \{ 0, \dots, n \} }{ b_i \nmid \lcm(S_i) }|$.    \end{itemize} 
    Note that $\cotype(S) \in \{0,1,\dots,n+1\}$ and that, if $S'$ is the normalization of $S$, then $\cotype(S) = \cotype(S')$.

	\end{definitions}

	\begin{definition}\label{ordering}
		Let $n \geq 2$.
		\begin{itemize}
			
			\item Given $S = (a_0,\dots,a_n) \in (\Nat^+)^{n+1}$ and $i \in \{0, \dots, n\}$, define $g_i(S) = \gcd(a_i, \lcm(S_i))$.
			
			\item Let $S = (a_0,\dots,a_n)$ and $S' = (a_0',\dots,a_n')$ be elements of $(\Nat^+)^{n+1}$ and let $i \in \{0, \dots, n\}$.
			We write $S \leq^i S'$ if and only if
			$$
			S_i = S'_i \quad \text{and} \quad g_i(S') \mid a_i \mid a_i'.
			$$
			We write $S <^i S'$ if and only if $S \leq^i S'$ and $S \neq S'$. Note that the relation $\leq^i$ is a partial order on $(\Nat^+)^{n+1}$.
		\end{itemize}
	\end{definition}

\subsection{$\Rat$-divisors}
This section establishes the notation and collects some results we will use when discussing $\Rat$-divisors. Whenever we discuss divisors and $\Rat$-divisors on a variety $X$, we assume implicitly that $X$ is normal. 
	
	\begin{definitions}  
        The group of \textit{$\Rat$-divisors}, denoted  $\Div(X, \Rat)$, is the group $\Div(X) \otimes_\Integ \Rat$.   A \textit{$\Rat$-divisor} $D$ is written as $D = \sum_{i \in I} a_i Y_i$ where each $Y_i$ is a prime divisor and $a_i \in \Rat$. If $D = \sum_{i \in I} a_i Y_i$ is a $\Rat$-divisor, then $\lfloor D \rfloor = \sum_{i \in I} \floor{a_i} Y_i$. Two $\Rat$-divisors $D$ and $D'$ are \textit{linearly equivalent} (we write $D \sim D'$) if there exists $f \in K(X)^*$ such that $D-D'$ is a principal divisor. Two $\Rat$-divisors are \textit{$\Rat$-linearly equivalent} (we write $D \sim_\Rat D'$) if there exists some $n \in \Nat^+$ such that $nD - nD'$ is a principal divisor. A $\Rat$-divisor $D = \sum_{i \in I} a_i Y_i$ is \textit{effective} if $a_i \geq 0$ for all $i$. It is \textit{reduced} if for all $i\in I$, either $a_i = 1$ or $a_i = 0$. The \textit{support} of a $\Rat$-divisor $D = \underset{i \in I}{\sum} a_i Y_i$ is $\Supp(D) = \underset{a_i \neq 0}{\bigcup} Y_i$.
		
		A \textit{$\Rat$-Cartier $\Rat$-divisor} is a $\Rat$-divisor $D \in \Div(X,\Rat)$ such that $nD$ is a Cartier divisor for some $n \in \Nat^+$. A \textit{$\Rat$-Cartier divisor} is an integral $\Rat$-Cartier $\Rat$-divisor. A variety $X$ is called \textit{$\Rat$-factorial} if every integral divisor on $X$ is $\Rat$-Cartier.
  
		Given $D \in  \Div(X,\Rat)$, the sheaf $\OSheaf_X(D)$ of $\OSheaf_X$-modules is defined by stipulating that if $U$ is a non-empty open subset of $X$ then
		$$
		\Gamma(U,\OSheaf_X(D)) = \{0\} \cup \setspec{ f \in K(X)^* }{ \div_U(f) + D|_U \ge 0 } .
		$$
		Note that if $D \in  \Div(X,\Rat)$, then by definition $\OSheaf_X(D)=\OSheaf_X(\lfloor D \rfloor)$.  
\end{definitions}
 \begin{notation}
\label{NormalCMCoincide}
A \textit{canonical divisor} on a normal variety $X$ of dimension $n$ is any integral divisor $K_X$ on $X$ such that the restriction of the sheaf $\OSheaf_X(K_X)$ to the regular locus $X_{\mathrm{reg}}$ of $X$ is isomorphic to the canonical sheaf $\omega_{X_{\mathrm{reg}}}=\bigwedge^n \Omega_{X_{\mathrm{reg}}/\mathbf{k}}$ of $X_{\mathrm{reg}}$. When $X$ is normal and Cohen-Macaulay, the dualizing and canonical sheaves of $X$ (denoted $\omega_X^o$ and $\omega_X$ respectively) are isomorphic and are also isomorphic to $\OSheaf_X(K_X)$ where $K_X$ is any canonical divisor of $X$. (See \cite[Remark 5.2]{KovacsStable}.)
 \end{notation}

        \begin{definitions} \label{ampleDefs} A Cartier divisor $D$ on a normal variety $X$ is called \textit{very ample} if $\OSheaf_X(D)$ is a very ample invertible sheaf; a $\Rat$-divisor $D$ on $X$ is \textit{ample} if there exists some $m \in \Nat^+$ such that $mD$ is a very ample Cartier divisor.  
        
        Assume in addition that $X$ is projective and let $D$ be a $\Rat$-Cartier $\Rat$-divisor. Then, $D$ is called \textit{nef} if $D \cdot C \geq 0$ for every irreducible curve $C$ in $X$. We say that $D$ is \textit{big} if $\Vol_X(D) = \limsup_{m \to \infty} \frac{h^0(X, \OSheaf_X(mD))}{m^d/d!} > 0$ and that $D$ is \textit{pseudoeffective} if it lies in the closure of the cone of effective divisors inside the N\'eron-Severi space $\NS(X) \otimes \Reals$.  Note that an ample divisor is pseudoeffective. 
\end{definitions}

Section \ref{Sec:Cotype0} also requires the following well-known theorem. 

\begin{theorem}[Nakai-Moishezon]\label{NakaiMoishezon}
   A $\Rat$-Cartier divisor $D$ on a normal projective surface $S$ is ample if and only if $D \cdot C > 0$ for every irreducible curve $C \in \Div(S)$.
\end{theorem}
\subsection{Quasismooth Weighted Complete Intersections}

We collect known results on quasismooth weighted complete intersections in weighted projective spaces. Although not strictly required, in order to be consistent with the our definition of ``variety" at the start of the article, we assume that the field $\bk$ below is algebraically closed. 
 
\begin{nothing}
Let $R = \bk_{w_0, \dots, w_n}[X_0, \dots, X_n]$ denote the graded polynomial ring where $n \geq 1$ and $\deg(X_i) = w_i$ for each $i = 0, \dots, n$. The weighted projective space $\PPP = \PPP(w_0, \dots , w_n) = \Proj(R)$ is said to be \textit{well-formed} if for each $i = 0,\dots,n$, we have $\gcd(w_0,\dots,w_{i-1},\hat{w_i},w_{i+1}, \dots , w_n) = 1$. Every weighted projective space is a projective variety and is isomorphic to a well-formed weighted projective space. Every weighted projective space is normal and Cohen-Macaulay (see \cite[Theorem 3.1A(c)]{beltramettiRobbiano}).

A \textit{weighted projective variety} $X$ is a closed subvariety of a weighted projective space. Whenever we write ``the variety $X \subseteq \PPP$", we mean that $X$ is a closed subvariety of $\PPP$. Since $\PPP(w_0, \dots , w_n)$ is a projective variety, every weighted projective variety is also a projective variety. A weighted projective variety $X \subseteq \PPP$ is \textit{well-formed} if both $\PPP$ is well-formed and $\codim_X(X \cap \Sing(\PPP)) \geq 2$. 

Let $I$ be a homogeneous prime ideal of the graded ring $R = \bk_{w_0,\dots,w_n}[X_0, \dots, X_n]$ and consider the closed subvariety $X_I = V_+(I)$ of $\PPP=\PPP(w_0,\dots,w_n) = \Proj(R)$.  Note that $X$ is isomorphic to $\Proj( R/I )$. If $f_1, \dots, f_k$ are homogeneous elements of $R$, we abbreviate $X_{\lb f_1, \dots, f_k \rb}$ by $X_{f_1, \dots, f_k}$. If $I$ is generated by a regular sequence $(f_1, \dots , f_k)$ of homogeneous elements of $S$ of respective degrees  $d_i = \deg(f_i)$ where $i = 1, \dots , k$, then $X_I$ is called a \textit{weighted complete intersection} of \textit{multidegree $(d_1, \dots, d_k)$}. In particular, if $k=1$, $f_1 = f$ and $d_1 = d$, we will say that $X_I = X_f$ is a \textit{weighted hypersurface} of \textit{degree $d$}. The closed subset $C_X = V(I) \subseteq \aff^{n+1}$ is called the  \textit{affine cone over $X$}; note that $C_X$ passes through the origin of $\aff^{n+1}$ and $C_X \isom \Spec( R/I )$ is an integral affine scheme.	The variety $X$ is called \textit{quasismooth} if $C_X$ is nonsingular away from the origin. When $X$ is well-formed and quasismooth, we have $\Sing(X) = X \cap \Sing(\PPP)$ (see \cite[p.185 and Proposition 8]{Dimca1986}).
  \end{nothing}

\begin{assumption}
    From this point onward, whenever we consider a quasismooth weighted complete intersection $X \subset \PPP$, we assume that $X$ is not contained in a hyperplane of $\PPP$. 
\end{assumption}
\begin{nothing} \label{cyclicQuotientSingularitiesPBNothing}
We collect some known results on well-formed quasismooth weighted complete intersections. Recall \cite[$\S 5$]{kollar_mori_1998} that a variety $X$ has  \textit{rational singularities} if for every resolution of singularities  $f : \tilde{X} \to X$  one has  $R^i f_* \OSheaf_{\tilde{X}} = 0$ for all $i > 0$. This condition is known to be equivalent to the property that $X$ is Cohen-Macaulay and $f_* \omega_{\tilde{X}} = \omega_X$.

    Let $\PPP = \PPP(w_0, \dots, w_n)$ be a well-formed weighted projective space. If $X \subseteq \PPP$ is a well-formed quasismooth weighted complete intersection then $X$ has at most cyclic quotient \cite[p.105]{iano-fletcher_2000} and hence rational singularities \cite[Corollary 7.4.10]{ishii2014introduction}; it follows that $X$ is Cohen-Macaulay. Since the quotient of a normal variety is normal, it follows that $X$ is normal. Finally, $X$ is $\Rat$-factorial by \cite[Proposition 5.15]{kollar_mori_1998}. 
\end{nothing}

We recall that if $X \subseteq \PPP(w_0, \dots, w_n)$ is a weighted complete intersection of multidegree $(d_1, \dots, d_k)$, then the \textit{amplitude} of $X$ is denoted by $\alpha = \sum_{i=1}^k d_i - \sum_{i=0}^n w_n$. A key property of well-formed quasismooth weighted complete intersections is that the adjunction formula holds for them, in the following form:  

\begin{theorem}\cite[Theorem 3.3.4] {dolgachev}\label{dualizing}
	Let $X \subseteq \PPP(w_0, \dots, w_n)$ be a well-formed quasismooth weighted complete intersection. Then $\omega_X \isom \omega_X^\circ \isom \OSheaf_X(\alpha)$ where $\alpha$ is the amplitude of $X$.  
\end{theorem}

\subsection{Singularities of surfaces, log pairs and singular del Pezzo surfaces}

We briefly recall some basic facts about singularities of log pairs. We refer the reader to \cite[Chapter 2]{kollar_mori_1998} for details. 
\begin{definitions}
A \textit{log pair} $(S,D)$ consists of a normal variety $S$ and an effective $\mathbb{Q}$-divisor $D=\sum_{i=1}^r a_iC_i$ such that the divisor $K_S+D$ is $\mathbb{Q}$-Cartier. Let $\pi:\tilde{S} \to S$ be a proper birational morphism. The morphism $\pi$ is called a \textit{log resolution} if $\tilde{S}$ is smooth, the exceptional locus of $\pi$ (denoted ${\rm Ex}(\pi)$) is a divisor and ${\rm Ex}(\pi) \cup \pi^{-1}(\Supp(D))$ is an SNC divisor on $\tilde{S}$. Denoting by $\tilde{C}_i$ the proper transform of $C_i$ and by $E_1, \dots, E_n$ the $\pi$-exceptional curves, we have 
        $$K_{\tilde{S}}  + \sum_{i = 1}^r a_i \tilde{C}_i  \sim_\Rat \pi^*(K_S + D)+\sum_{j=1}^n b_j E_j$$
        for some rational numbers $b_1, \dots, b_n$.

A log pair $(S,D)$ is  called  \textit{canonical}, (resp. \textit{purely log-terminal}),  (resp.  \textit{log-canonical}) at a point $p$ if $a_i\leq 1$ for every $i$ such that $p\in C_i$ and for every log resolution $\pi:\tilde{S}\to S$ and every $j \in \{ 1, \dots, n\}$ such that $\pi(E_j)=p$, the coefficients $b_j$ satisfy     
$b_j\geq 0$ (resp. $b_j> -1$), (resp. $b_j\geq -1$). Given a log resolution $\pi$, the coefficient $b_j$  is called the \textit{discrepancy} of the exceptional curve $E_j$.

        The pair $(S,D)$ is call \textit{canonical}, (resp. \textit{purely log-terminal}), (resp. \textit{log-canonical}) if it is so at every point $p \in S$. A pair $(S,D)$ is called \textit{Kawamata log-terminal} (klt) if it is purely log-terminal and $\lfloor D \rfloor=0$. A normal variety $S$ is said to have \textit{canonical} (resp. klt) singularities if if the log pair $(S,0)$ is canonical  (resp. klt). 
These notions are known to be independent 
the choice of log resolution and can therefore be verified on the minimal log resolution of the log pair $(S,D)$. See \cite[Corollary 2.13]{kollar_2013}. 
\end{definitions}

In the special case where $S$ is a surface, it is known that klt singularities are finite quotient singularities and that among these, canonical singularities are exactly the \textit{Du Val ADE singularities}. 

\begin{example}\cite[p.286]{hacking2017flipping}\label{Cyclicdiscrepancy}
    Let $S$ be a normal projective surface with $k$ singular points $P_1, \dots, P_k$. Assume each $P_i$ is a cyclic quotient singularity of type $\frac{1}{n_i}(1,1)$ where $n_i \geq 2$ and let $\pi : \tilde{S} \to S$ be the minimal resolution of singularities of $S$. Then $E_i=\pi^{-1}(P_i)$ is isomorphic to $\mathbb{P}^1$ and has self-intersection number $-n_i$ in $\tilde{S}$. Moreover, in the ramification formula 
    $$K_{\tilde{S}} = \pi^*(K_S) + \sum_{i=1}^k b_i E_i$$ for $\pi$, $b_i = -1 + \frac{2}{n_i}$ for each $i = 1, \dots, k$. In particular, $P_i$ is a klt singularity which is canonical if and only if $n_i=2$.
\end{example}

\begin{definition}
 A \textit{del Pezzo surface}  $S$ is a normal projective surface over $\bk$ with at most quotient singularities such that $-K_S$ is an ample %(resp. nef and big) 
 $\Rat$-Cartier divisor. The \textit{degree} of a %(weak) 
 del Pezzo surface is the self-intersection number of its canonical divisor. 
\end{definition}

We also require the following lemma, used in Section \ref{subsub23530}.

\begin{lemma}$($\cite[Lemma A.3.]{cheltsov2020cylinders}$)$\label{lem:multlc}
		Let $S$ be a normal surface with at most quotient singularities and let $D$ be an effective non-zero $\Rat$-divisor on $S$. Let $p$ be a regular point of $S$. If $(S,D)$ is not log-canonical at $p$, then $\mult_p(D) > 1$.		
\end{lemma}

\subsection{Cylinders and Polar Cylinders}\label{subsec:cylinders}

\begin{definition}
		A scheme $U$ is a \textit{cylinder} if $U \isom C \times {\aff^1}$ for some affine scheme $C$. A scheme $X$ \textit{contains a cylinder} if there exists a non-empty open set $U \subseteq X$ such that $U$ is a cylinder. Note that $U$ is affine. 
\end{definition}

\begin{definition}\label{def:HPolarCylinder}		
		Let $H$ be a $\Rat$-divisor on a projective normal variety $X$ over $\bk$. Let $U \subseteq X$ be a cylinder of $X$. The cylinder $U \subseteq X$ is called \textit{$H$-polar}  if $U = X \setminus \Supp(D)$ for some effective $\Rat$-Cartier $\Rat$-divisor $D \in \Div(X, \Rat)$ such that $D \sim_\Rat H$. In the special case where $H \sim -K_X$, any $H$-polar cylinder $U$ is called an \textit{anti-canonical polar cylinder}.  
\end{definition}

\begin{remark}\label{HH'polar}
    Let $H$ and $H'$ be $\Rat$-divisors on a projective normal variety $X$.  Assume that there exist $q,q' \in \Rat^+$ such that $qH \sim q'H'$. Then a cylinder $U$ is $H$-polar if and only if it is $H'$-polar.
\end{remark}

\subsubsection*{Demazure's Construction and Polar Cylinders}

\begin{definition}
	Let $B = \bigoplus_{i \in \Nat} B_i$ be an $\Nat$-graded integral domain. An element $\xi \in \Frac B$ is \textit{homogeneous} if $\xi = \frac{a}{b}$ for some homogeneous elements $a,b \in B$ with $b \neq 0$. If $\xi$ is homogeneous, then its \textit{degree} is $\deg(\xi) = \deg(a)-\deg(b)$. 
\end{definition}

The following is a special case of the Theorem below Section 3.5 in \cite{Demazure}. 
\begin{theorem}\label{Demazure}
	Let $B =\bigoplus_{n \in \Nat} B_n$ be an $\Nat$-graded normal domain that is finitely generated over $\bk$ and such that $e(B) = 1$ and $\height(B_+) > 1$.  Let $X = \Proj B$. Then, there exists a homogeneous element $T$ of $\Frac B$ of degree 1, and for each such $T$, there exists a unique $\Rat$-divisor $H$ of $X$ such that $B_n = \text{\rm H}^0(X,\OSheaf_X(nH)) T^n$ for all $n \in \Nat$. Moreover, $H$ is ample and $\OSheaf_X(n) \isom \OSheaf_X(nH)$ for all $n \in \Integ$. 	
\end{theorem}

\begin{caution}
	In Theorem \ref{Demazure}, $nH$ is a $\Rat$-divisor of $X$, so $\OSheaf_X( nH )$ is an abbreviation for $\OSheaf_X( \lfloor nH \rfloor )$
	by definition. Thus, Theorem \ref{Demazure} asserts that
	$\OSheaf_X(n) \isom \OSheaf_X( \lfloor nH \rfloor)$ and $B_n = {\rm {H}}^0( X, \OSheaf_X(\lfloor nH \rfloor)) T^n$  for all $n \in \Nat$.
\end{caution}

\begin{nothing}(\cite[Sec. 5.4]{ChitayatDaigleCylindricity})\label{Pkcnbvc9w3eidjojf0q9w} Let $B$ be an $\Nat$-graded Noetherian normal domain such that the prime ideal $B_+ = \underset{i>0}{\bigoplus} B_i$ has height greater than $1$.
	Let ${\Omega} = \Spec B$ and $X = \Proj B$. We shall now define a $\Rat$-linear map $D \mapsto D^*$ from $\Div(X, \Rat)$ to $\Div({\Omega},\Rat)$.
	
	Let $K({\Omega})$ and $K(X)$ be the function fields of ${\Omega}$ and $X$ respectively. Let $X^{(1)}$ be the set of homogeneous prime ideals of $B$ of height $1$. Since $\height{B_+}>1$, $X^{(1)} = \setspec{ x \in X }{ \dim\OSheaf_{X,x}=1 }$.
	For each $\pgoth \in X^{(1)}$, $B_\pgoth \supset B_{(\pgoth)}$ is an extension of discrete valuations rings; let $e_\pgoth$ denote the ramification index of this extension.	Then $e_\pgoth \in \Nat\setminus \{0\}$. If $v^X_\pgoth : K(X)^* \to \Integ$ and 
	$v^{\Omega}_\pgoth : K({\Omega})^* \to \Integ$ denote the normalized\footnote{The word ``normalized'' means that the maps $v^X_\pgoth$ and $v^{\Omega}_\pgoth$ are surjective.}
	valuations of $B_{(\pgoth)}$ and $B_\pgoth$ respectively,
	then $v^{\Omega}_\pgoth (\xi) = e_\pgoth v^X_\pgoth(\xi)$ for all $\xi \in K(X)^*$.
	Let $C_\pgoth^X$ (resp.~$C_\pgoth^{\Omega}$) denote the closure of $\{ \pgoth \}$ in $X$ (resp.\ in ${\Omega}$).
	Then $C_\pgoth^X$ (resp.~$C_\pgoth^{\Omega}$) is a prime divisor of $X$ (resp.\ of ${\Omega}$),
	and every prime divisor of $X$ is a $C^X_\pgoth$ for some $\pgoth \in X^{(1)}$. We define $( C_\pgoth^X )^* = e_\pgoth C_\pgoth^{\Omega}$ for each $\pgoth \in X^{(1)}$,
	and extend linearly to  a $\Rat$-linear map $\Div(X, \Rat) \to \Div({\Omega},\Rat)$, $D \mapsto D^*$.
	It can be verified that the linear map $D \mapsto D^*$ is injective and has the following two properties:
	\begin{enumerate}
		\item[\rm(a)]  $\big( \div_X(\xi) \big)^* = \div_{\Omega}(\xi)$ for all $\xi \in K(X)^*$;
		\item[\rm(b)] if $f$ is a nonzero homogeneous element of $B$ and $D \in \Div(X, \Rat)$ satisfies $D^* = \div_{\Omega}(f)$, then $D \geq 0$ and $\Supp(D) = V_+(f)$.
	\end{enumerate}
\end{nothing}

\begin{lemma}  \label {7fr6d543sqderfvuwR8u390rucyq74etf}
	Let the assumptions and notation be as in Theorem \ref{Demazure}.
	\begin{enumerate}
		
		\item[\rm(a)](\cite[p.52]{Demazure}) If $T$ and $H$ are as in Theorem \ref{Demazure} then $H^* = \div_{\Omega}(T)$.
		
		\item[\rm(b)](\cite[Remark 5.10 (d)]{ChitayatDaigleCylindricity}) Let $T_1,T_2$ be homogeneous elements of $\Frac B$ of degree $1$ and for $i=1,2$
		let $H_i$ be the $\Rat$-divisor of $X$ that corresponds to $T_i$ as in Theorem \ref{Demazure}.
		Then $T_1/T_2 \in K(X)^*$ and $\div_X(T_1/T_2) = H_1 - H_2$. In particular, $H_1 \sim H_2$.
		
	\end{enumerate}
\end{lemma}

The following is a special case of \cite[Lemma 5.20(a)]{ChitayatDaigleCylindricity}. One direction was originally shown in \cite[Remark 1.14]{KishimotoProkhorovZaidenberg}; we suspect both directions were likely known at the time.

\begin{lemma}\label{HpolarPrincipal}
	Let the assumptions and notation be as in Theorem \ref{Demazure}. Fix some choice of homogeneous $T \in \Frac B$ of degree 1 as well as its corresponding $\Rat$-divisor $H$. Then, a cylinder $U$ of $X$ is $H$-polar if and only if there exist $n  \geq 1$ and $h \in B_n \setminus \{0\}$ such that $U = D_+(h)$.
\end{lemma}

\subsubsection*{Cylinders and $\mathbb{P}^1$-fibrations on Normal Projective Surfaces}
\label{P1fibrations}

We now recall basic geometric consequences of the existence of cylinders on normal projective surfaces. See \cite[Section 2.1]{cheltsov2020cylinders}.

    \begin{definition}
        A surjective morphism $\phi: V \to B$ between projective varieties $V$ and $B$ is a called \textit{$\PPP^1$-fibration} if a general closed fiber of $\phi$ is isomorphic to $\PPP^1$. 
    \end{definition}

\begin{nothing}\label{cylinderSetup}$($\cite[p.49]{cheltsov2020cylinders}$)$
Assume throughout this subsection that $U \isom Z \times \Aff^1$ is a cylinder contained in a normal projective surface $S$ (so $Z$ is an affine smooth curve). Note that since $U$ is smooth and affine and $S$ is proper over $\bk$, the $i = 0$ case of \cite[Theorem 4.3]{hartshornedeRham} implies that $S \setminus U$ is connected. The projection $\mathrm{pr}_{Z}:U\to Z$ extends to a rational map $\rho:S \dashrightarrow\bar{Z}$ where $\bar{Z}$ is the smooth projective model of $Z$ and the general fibers of $\rho$ in the rational sense are the closures in $S$ of the fibers of $\mathrm{pr}_{Z}$. Since the fibers of $\mathrm{pr}_{Z}$ are isomorphic to $\mathbb{A}^1$, their closures have a unique point at infinity. This implies that either $\rho:S \dashrightarrow\bar{Z}$ is an everywhere defined $\mathbb{P}^1$-fibration having one of the irreducible components of $S\setminus U$ as a section or it is a strictly rational map with a unique proper base point on $S$, equal to the intersection of these closures.  By resolving the singularities of $S$ as well as, if any, the indeterminacy of the rational map $\rho$ (by blowing-up its unique proper base point and then all the subsequent infinitely near points), we obtain the following commutative diagram where $\phi:W \to \bar{Z}$ is a $\PPP^1$-fibration and $W$ is smooth. 

            $$
		  \xymatrix{
			U \isom \aff^1 \times Z \ar@{^{(}->}[r] \ar[d]_{\mathrm{pr}_{Z}}  & S \ar@{-->}[d]^\rho & W \ar[l]_\alpha \ar[dl]^\phi \\		
			Z \ar@{^{(}->}[r] & \bar{Z} & &
		  }
		  $$
Let $C_1, \dots,C_n$ be the irreducible curves in $S$ such that $S \setminus U = \cup_{i = 1}^n C_i$.  Let $E_1, \dots, E_r,$ denote the exceptional curves of $\alpha$ and let $\tilde{C_i}$ denote the proper transform of $C_i$ in $W$. Then exactly one curve among $\tilde{C}_1, \dots, \tilde{C}_n, E_1, \dots, E_r$ (say either $\tilde{C}_1$ or $E_r$) is a section of $\phi$ and all other curves $\tilde{C}_i$ and $E_j$ are contained in closed fibers of $\phi$. Moreover, $\rho$ is a morphism if and only if $C_1$ is a section of $\rho$. 
\end{nothing}
  
        The following lemma is almost identical to \cite[Lemma 3.7]{cheltsov_park_won_2016} and the proof is essentially the same.  
        \begin{lemma} \label{cylinderSDNotLogCanonical} 
        Let $D=\sum_{i = 1}^n a_i C_i$ be an effective anti-canonical $\Rat$-divisor on normal projective surface $S$ with quotient singularities such that $U=X\setminus \mathrm{Supp}(D)$ is a cylinder. Then the log pair $(S,D)$ is not log-canonical. Moreover, if the map $\rho :S \dashrightarrow \bar{Z}$ as in \ref{cylinderSetup} has a unique proper base point $s\in S$, then the log pair $(S,D)$ is not log-canonical at $s$.
	\end{lemma}
	
	\begin{proof}
		We use the notation of \ref{cylinderSetup} throughout and assume further that $\alpha:W\to S$ is a resolution of singularities such that the union of the proper transform $D_W$  of $D$ on $W$ and of the exceptional locus $E=\overset{r}{\underset{j=1}{\bigcup}} E_i$ is an SNC divisor on $W$. Since
		$D$ is an effective anti-canonical $\Rat$-divisor on $S$, we have $K_{S}+D\sim_\Rat 0$, so the ramification formula for
		$\alpha$ reads 
		$$
		K_{W}+D_{W}\sim_{\mathbb{Q}}\alpha^{*}(K_{S}+D)+\sum_{j=1}^{r}b_{j}E_{j}\sim_{\mathbb{Q}}\sum_{j=1}^{r}b_{j}E_{j}.
		$$
         If $\rho:S\to \bar{Z}$ is a morphism, hence a $\mathbb{P}^1$-fibration, then there exists a unique component of $D$, say $C_1$, whose proper transform $\tilde{C}_1$ in $W$ is a section of $\phi$ and the proper transforms of all other irreducible components of $D$ as well as the exceptional curves of $\alpha$ are contained in closed fibers of $\phi$. For a general fiber $L$ of $\phi$, we obtain the equality $$0=(\sum_{j=1}^{r}b_{j}E_{j})\cdot L=(K_W+D_W)\cdot L=-2+a_1$$ which implies that $a_1=2$. Hence,
         \begin{equation}\label{intermediate}
         \text{for every point $p$ on $C_1$, the log pair $(S,D)$ is not log-canonical at $p$.} 
         \end{equation}
         If $\rho$ is not a morphism, then the irreducible components of $D_W$ are contained in closed fibers of $\phi$. Consequently, a general fiber $L\cong\mathbb{P}^{1}$ of $\phi:W\to\PPP^1$
		is a $(0)$-curve that does not intersect $D_{W}$ and (by \ref{cylinderSetup}) intersects exactly one of the curves $E_{j}$, say $E_r$,  transversely at a unique point. We have 
		\[
		-2=K_{W}\cdot L=(K_{W}+D_{W})\cdot L=(\sum_{j=1}^{r}b_{j}E_{j})\cdot L=b_{r}.
		\]
		Thus, $b_{r} = -2 < -1$ which shows that the log pair $(S,D)$ is not
		log-canonical at the point $s$. 
    \end{proof}

 The following classifies precisely which del Pezzo surfaces with at most Du Val singularities admit anti-canonical polar cylinders and is used in Section \ref{Sec:Cotype0}. 
	
\begin{theorem}$($\cite[Theorem 1.5]{cheltsov_park_won_2016}$)$\label{thm:CPWantiCanonical}
Let $S$ be a del Pezzo surface of degree $d$ with at most Du Val singularities. The surface $S$ does not admit a $-K_{S}$-polar cylinder if and only if one of the following hold:
    \begin{enumerate}[\rm(1)]
		\item $d = 1$ and $S$ allows only singular points of types $A_1, A_2, A_3, D_4$ if any;
		\item $d = 2$ and $S$ allows only singular points of type $A_1$ if any;
		\item $d = 3$ and $S$ allows no singular point.
\end{enumerate}
\end{theorem}

\section{Reductions of the Conjecture}
\label{sec:reduction}
In this section, we show that Main Conjecture 
holds for all graded rings $B_{a_0,\ldots,a_n}$ provided that it holds for those whose associated quasismooth hypersurfaces 
$$\Proj(B_{a_0, \dots, a_n}) = \{\sum_{i = 0}^n X_i^{a_i}=0\} \subset \PPP = \PPP(w_0, \dots, w_n)$$ are well-formed. The proof builds in part on a characterization of these rings as being precisely those for which the $(n+1)$-tuple $(a_0, \ldots, a_n)$ has cotype $0$. We then recall, following a general principle introduced in \cite{KishimotoProkhorovZaidenberg},the relationship between the rigidity of $B_{a_0,\ldots, a_n}$ and the non-existence of certain polar cylinders in  $\Proj(B_{a_0, \dots, a_n})$.

\subsection{Reduction to well-formed hypersurfaces}

\begin{nothing}\label{PBGrading} Let $n \geq 2$ and let $S = (a_0, \dots, a_n) \in (\Nat^+)^{n+1}$. Let $f = X_0^{a_0} + \dots  + X_n^{a_n} \in \bk[X_0, \dots , X_n]$. Let $L = \lcm(a_0, \dots, a_n)$, let $\deg(X_i) = w_i = L / a_i$ for each $i \in \{0, 1, \dots, n\}$ and note that $\gcd(w_0, \dots, w_n) = 1$. Then $f$ is homogeneous of degree $L$, $B_{a_0,\ldots, a_n} = \bk_{w_0, \dots, w_n}[X_0, \dots , X_n] / \lb f \rb$ is an $\Nat$-graded ring, and $\deg(x_i) = w_i$ for every $i = 0, \dots, n$.
 Since $B_{a_0,\ldots, a_n}$ is regular in codimension one and Cohen-Macaulay (since it is a hypersurface), it follows from Serre's Normality Criterion that $B_{a_0,\ldots, a_n}$ is normal. Consequently, the variety $$X_f = \Proj(\bk_{w_0, \dots, w_n}[X_0, \dots , X_n] / \lb f \rb) = \Proj(B_{a_0, \dots, a_n})$$ is a normal quasismooth weighted projective hypersurface of degree $L$ in the weighted projective space $\mathbb{P}(w_0,\ldots,w_n)$.  
\end{nothing}

\begin{definition}$($\cite[Section 3]{ChitayatDaigleCylindricity} $)$\label{saturationIndex}
		Let $B = \bigoplus_{i \in \Integ} B_i$ be a $\Integ$-graded domain. The \textit{saturation index of $B$} is defined as $ e(B) = \gcd\setspec{i \in \Integ}{B_i \neq 0}.$ The graded ring $B$ is \textit{saturated in codimension 1} if $e(B/\pgoth) = e(B)$ for every homogeneous height one prime ideal $\pgoth$ of $B$. 
	\end{definition} 
	\begin{remark}$($\cite[Example 3.16]{ChitayatDaigleCylindricity} $)$\label{PBsatCodim1}
		Let $n \geq 2$ and let $B = B_{a_0, \dots, a_n}$. Then the following are equivalent:
		
		\begin{itemize}
			\item $B$ is saturated in codimension 1; 
			\item $\cotype(a_0, \dots, a_n) = 0$. 
		\end{itemize}
	\end{remark}

\begin{proposition}\label{PBWellFormed}\cite[Proposition 4.6]{chitayat2023rationality}
		Let $f = X_0^{a_0} + \dots + X_n^{a_n}$ where $n \geq 2$ and $a_i \geq 1$ for all $i$. Let $L = \lcm(a_0, \dots, a_n)$ and let $w_i = L / a_i$. Then, the weighted hypersurface $X_f \subset \PPP(w_0, \dots, w_n)$ is well-formed and quasismooth if and only if $\cotype(a_0, \dots , a_n) = 0$.
\end{proposition}

		We will now show that if the Main Conjecture
  holds for rings $B_{a_0,\ldots, a_n}$ with $\cotype(a_0,\ldots, a_n)=0$ then it holds in general. 

        \begin{notation}
    Let $B$ be a $\Integ$-graded ring and let $A$ be a graded subring of $B$. Define $I_A = \setspec{i \in \Integ}{A_i \neq 0}$ and define $\Integ(A)$ to be the subgroup of $\Integ$ generated by $I_A$.
    \end{notation}
 
		\begin{definition} $($\cite[Definition 5.1, $G = \Integ$]{DFM2017} $)$
			Let $B$ be a $\Integ$-graded ring. A nonzero homogeneous
			element $x$ of $B$ is \textit{$\Integ$-critical} if there exists a graded subring $A \subset B$ such that $\Integ(A) \neq \Integ(B)$ and $B = A[x]$. 
		\end{definition}
		
		\begin{lemma}$($\cite[Thoerem 6.2, $G = \Integ$]{DFM2017} $)$ \label{ZcriticalLocalSlice}
			Let $B$ be a $\Integ$-graded integral domain, and let $D \in \lnd(B)$ be homogeneous. For every $\Integ$-critical element $x \in B$, $D^2(x) = 0$.
		\end{lemma}
		
		\begin{lemma}$($\cite[Lemma 3.1]{freudenburg2013} $)$\label{quasiextensionLemma}
			Assume $R$ is a $\bk$-domain, $f \in R$, $n \geq 2$, and $f+Z^n$ is a prime element of $R[Z]$. If $|f|_R \leq 1$, then $B = R[Z] / \lb Z^n + f \rb$ is not rigid. 
		\end{lemma} 
  
		\begin{lemma}\cite[Lemma 1.3.6]{ChitayatThesis}\label{ufh873268guj83497yg}
			Let $n \geq 2$, let $S = (a_0,\dots,a_n) \in (\Nat^+)^{n+1}$,
			and consider 
			$$
			B_S = \bk[X_0,\dots,X_n] / \lb X_0^{a_0} + \cdots + X_n^{a_n} \rb = \bk[x_0,\dots,x_n] .
			$$
			Let $m \in \Nat^+$, let $S^m = (a_0, \dots, a_{n-1}, m a_n)$ and write 
			$$B_{S^m} = \bk[Y_0, \dots, Y_n] / \lb Y^{a_0} + \dots + Y_{n-1}^{a_{n-1}} + Y_n^{ma_n} \rb = \bk[y_0, \dots, y_n].$$ Then there is a $\bk$-isomorphism $B_S[Z] / \lb Z^m - x_n \rb \isom B_{S^m}$ where $Z$ is an indeterminate over $B_S$.
		\end{lemma}
		
		\begin{proposition}\label{cotype1method}
			Let $n \geq 2$ and suppose $(a_0,\dots,a_n) \in (\Nat^+)^{n+1}$ satisfies
			\begin{itemize}
				
				\item $a_n \nmid \lcm(a_0, \dots, a_{n-1})$,
				
				\item there exists $m \in \Nat^+$ such that $B_{a_0, \dots, a_{n-1}, m a_n}$ is rigid.
				
			\end{itemize}
			Then $B_{a_0, \dots, a_{n-1}, a_n}$ is rigid.
		\end{proposition}
		
		\begin{proof}
		  Let $R = B_{a_0, \dots, a_n} = \bk[x_0, \dots, x_n]$. Let $A = \bk[x_0, \dots, x_{n-1}]$ and observe that $R = A[x_n]$. By \ref{PBGrading}, $\Integ(R) = \Integ$. Let $L' = \lcm(a_0, \dots, a_{n-1})$, and let $L = \lcm(a_n,L')$. Since by assumption $a_n \nmid L'$, $\Integ(A) = (L / L')\Integ \subset \Integ$. It follows that $x_n$ is a $\Integ$-critical element of $R$.
			Let $m \in \Nat^+$ be such that $B_{a_0, \dots, a_{n-1}, m a_n}$ is rigid, let $R' = R[Z]/ \lb Z^m-x_n \rb$ and note that Lemma \ref{ufh873268guj83497yg} implies that $R' \isom B_{a_0, \dots, a_{n-1}, m a_n}$.
			
			Assume that $R$ is not rigid. By Proposition \ref{homogenization}, there exists $D \in \lnd(R)$ which is nonzero and homogeneous;
			since $x_n$ is a $\Integ$-critical element of $R$, Lemma \ref{ZcriticalLocalSlice} implies that $D^2(x_n) = 0$.  In particular,  $|x_n|_R \leq 1$.
			By Lemma \ref{quasiextensionLemma}, $R'$ is not rigid, hence $B_{a_0, \dots, a_{n-1}, m a_n}$ is not rigid. This is a contradiction, so $R$ is rigid.
		\end{proof}
		
            \begin{notation}\label{UTplusminus}
    Let $L = \lcm(a_0, \dots, a_n)$ and let $\alpha = L - \sum_{i = 0}^n \frac{L}{a_i}$. Define the following sets
	\begin{align*}	
             \Gamma_n &= \setspec{(a_0,\dots, a_n) \in (\Nat^+)^{n+1}}{\text{$\min(a_0,\dots,a_n)>1$ and at most one $i$ satisfies $a_i=2$}} \\
		\Gamma_n^+ &= \setspec{(a_0,\dots,a_n) \in \Gamma_n}{\cotype(a_0,\dots,a_n) = 0 \text{ and } \alpha \geq 0}\\	
		\Gamma_n^- &= \setspec{(a_0,\dots,a_n) \in \Gamma_n}{\cotype(a_0,\dots,a_n) = 0 \text{ and } \alpha < 0}.
	\end{align*} 
and consider the statements:	
			\begin{align*}
				P(n): & \text{ $B_{a_0, \dots, a_n}$ is rigid for all $(a_0, \dots, a_n) \in \Gamma_n$ }. \\
				P(n,i): & \text{ $B_{a_0, \dots, a_n}$ is rigid for all $(a_0, \dots, a_n) \in \Gamma_n$ satisfying $\cotype(a_0, \dots, a_n) = i$.}
			\end{align*}
  \end{notation}
		The following appears in \cite{Chitayat_Daigle_2019} with slightly different notation. For $S = (a_0,\dots,a_n) \in (\Nat^+)^{n+1}$, we define $B_S = B_{a_0, \dots, a_n}$. 
	 	
		\begin{proposition} \label{cljbo238ecv0}\cite[Proposition 4.9 (a)]{Chitayat_Daigle_2019}
			Let $n\geq 2$, let $S,S' \in (\Nat^+)^{n+1}$ and suppose  $S' \leq^i S$ for some $i \in \{0,\dots,n\}$. If $B_{S'}$ is rigid then $B_S$ is rigid.
		\end{proposition}

		\begin{theorem} \label{reduction}
			Let $n \geq 3$. If $P(n-1)$ and $P(n,0)$ hold, then $P(n)$  holds.			
		\end{theorem}
		\begin{proof}
		 Suppose $S = (a_0, \dots, a_n) \in \Gamma_n$ and assume that $P(n-1)$ and $P(n,0)$ hold. We must show that $B_S$ is rigid. If $\cotype(S) = 0$ we are done, since $P(n,0)$ holds by assumption. Assume henceforth that $\cotype(S) \geq 1$. 
			
			Suppose $\cotype(S) \geq 2$. By contradiction, assume that there exists $D \in \lnd(B_S) \setminus \{0\}$.  Without loss of generality, by Proposition \ref{homogenization} together with Theorem \ref{p0cfi2k309cbqp90ws} (d) we may assume $D$ is homogeneous and irreducible. For each $i \in \{0, \dots, n\}$, let $w_i = \deg(x_i)$ be as in \ref{PBGrading}. Let $H_i = \lb w_0, \dots, w_{i-1}, \hat{w_i}, w_{i+1}, \dots, w_n \rb \subseteq \Integ$ and $S_i = (a_0, \dots, a_{i-1}, \hat{a_i}, a_{i+1}, \dots a_n)$.	 Since $\cotype(S) \geq 2$, there exist distinct $j,k \in \{0, \dots, n\}$ such that $H_j \subset \Integ$ and $H_k \subset \Integ$. By \cite[Corollary 6.3 (b)]{DFM2017}, either $x_j \in \ker(D)$ or $x_k \in \ker(D)$. Without loss of generality, we may assume $j=n$ or $k=n$, so that $x_n \in \ker(D)$. Then, since $D$ is irreducible, $D$ induces a nonzero locally nilpotent derivation on $ B_S/ \lb x_n \rb  \isom B_{S_n}$. But since $S \in \Gamma_n$, it follows that $S_n \in \Gamma_{n-1}$. This is a contradiction since $P(n-1)$ holds. So $B_S$ is rigid when $\cotype(S) \geq 2$.

			Finally, assume $\cotype(S) = 1$. Then, up to permuting the $a_i$, we may arrange that $a_n \nmid \lcm(a_0, \dots, a_{n-1})$. Let $L = \lcm(a_0, \dots, a_{n-1})$. By Proposition \ref{cotype1method}, it suffices to prove that $B_{a_0, \dots, a_{n-1}, a_nL}$ is rigid. We have $(a_0, \dots, a_{n-1}, L) <^n (a_0, \dots, a_{n-1}, a_nL)$ and $\cotype(a_0, \dots, a_{n-1}, L) = 0$. Since $L > 2$, $(a_0, \dots, a_{n-1}, L) \in \Gamma_n$. By assumption, $P(n,0)$ holds so $B_{a_0, \dots, a_{n-1}, L}$ is rigid and hence $B_{a_0, \dots, a_nL}$ is rigid by Proposition \ref{cljbo238ecv0}.
		\end{proof}	

Combining Theorem \ref{reduction} with the fact that the base case  $P(2)$ holds by \cite[Lemma 4]{Kali-Zaid_2000}, we derive the following corollary: 
\begin{corollary} \label{reductionCor} The Main Conjecture 
 holds in dimension $n=3$ if and only if it holds for well-formed Pham-Brieskorn threefolds $X_{a_0,a_1,a_2,a_3}$, that is, if and only if $B_{a_0,a_1,a_2,a_3}$ is rigid for every $(a_0,a_1,a_2,a_3)\in \Gamma_3^+ \cup \Gamma_3^-$. 
  \end{corollary}

Arguing by induction, Theorem \ref{reduction} reduces the study of the Main Conjecture to the natural class of graded rings $B_{a_0,\ldots,a_n}$ where $(a_0,\ldots, a_n)\in \Gamma_n$ for which the associated quasismooth hypersurface $\Proj(B_{a_0, \dots, a_n})$ is well-formed, namely $(a_0, \dots, a_n) \in \Gamma_n^+ \cup \Gamma_n^-$. More formally, we obtain:

\begin{corollary}
    If $P(n,0)$ is true for all $n \geq 3$ then the Main Conjecture holds.  
\end{corollary}

\subsection{Reduction to the non-existence of polar cylinders}

The following is special case of \cite[Theorem 1.2]{ChitayatDaigleCylindricity}, which in turn is a generalization of \cite[Theorem 0.6]{KishimotoProkhorovZaidenberg}. 
\begin{theorem} \label{edh83yf6r79hvujhxu6wrefji9e} 
		Let $B = \bigoplus_{i \in \Nat}B_i$ be an $\Nat$-graded affine $\Comp$-domain such that the transcendence degree of $B$ over $B_0$ is at least $2$. The following are equivalent.
		\begin{enumerate}[\rm(a)]
			
			\item There exists $d \geq 1$ such that $B^{(d)}$ is not rigid.
			
			\item There exists a homogeneous element $h \in B \setminus \{0\}$ of positive degree such that the open subset $D_+(h)$ of $\Proj B$ is a cylinder.
			
		\end{enumerate}
	
	\noindent Moreover, if $B$ is normal and is saturated in codimension 1 then the above conditions are equivalent to 
	\begin{enumerate}[\rm(c)]
		\item $B$ is not rigid.
	\end{enumerate}	

	\end{theorem}

	\begin{corollary}\label{BdRigid}
		Let $n \geq 2$ and suppose $\cotype(a_0, \dots, a_n) = 0$. The following are equivalent: 
		
		\begin{itemize}
			\item $B_{a_0, \dots, a_n}$ is rigid; 
			\item $(B_{a_0, \dots, a_n})^{(d)}$ is rigid for all $d \in \Nat^+$. 
		\end{itemize} 
	\end{corollary}

	\begin{proof}
		It suffices to check that $B_{a_0, \dots, a_n}$ satisfies the assumptions of Theorem \ref{edh83yf6r79hvujhxu6wrefji9e}. The only non-trivial things to check are normality  and saturation in codimension 1. Normality is shown in \ref{PBGrading} and Remark \ref{PBsatCodim1} implies that $B_{a_0,\dots, a_n}$ is saturated in codimension 1. 
	\end{proof}

	\begin{remark}
		The assumption that $\cotype(a_0, \dots, a_n) = 0$ is necessary for Corollary \ref{BdRigid} to hold. We will see later that $B_{2,3,3,4}$ is rigid whereas $(B_{2,3,3,4})^{(2)} \isom B_{2,3,3,2}$ is not rigid (as discussed in the Introduction).  
	\end{remark}

The following appears as Lemma 4.1.8 in \cite{ChitayatThesis}. The proof given here is simpler. 

\begin{lemma}\label{PBramification}
	With the notation of \ref{Pkcnbvc9w3eidjojf0q9w}, let $B = B_{a_0, \dots, a_n}$ where $\cotype(a_0, \dots, a_n) = 0$ and let $\pgoth_i = \lb x_i \rb \lhd B$. For each $i = 0, \dots, n$, $\pgoth_i \in X^{(1)}$ and $e_{\pgoth_i} = 1$. 
\end{lemma}

\begin{proof}
    By Remark \ref{PBsatCodim1}, $B$ is saturated in codimension 1. By  \cite[Corollary 9.4]{daigle2023rigidity}, $e_{\qgoth} = 1$ for all $\qgoth \in X^{(1)}$. In particular, the result is true when $\qgoth = \pgoth_i$.
\end{proof}

\begin{theorem} \label{PBCanonical2}
	Let $n \geq 2$ and consider $(a_0,\dots, a_n) \in (\Nat^+)^{n+1}$ of cotype $0$.
	Let $B = B_{a_0,\dots, a_n}$, let $X = \Proj(B)$, and let $\alpha$ be the amplitude of $X$. Then,
	\begin{enumerate}[\rm(a)]
		
		\item $\omega_X \isom \OSheaf_X(\alpha)$.
		
		\item Let $T$ be a homogeneous element of $\Frac B$ of degree 1 and let $H$ be the unique $\Rat$-divisor of $X$
		determined by $T$ as in Theorem \ref{Demazure}. Then $H \in \Div(X)$, $K_X \sim \alpha H$ and $K_X$ is $\Rat$-Cartier.
		
		\item Assume that $\alpha \neq 0$ and define $s = \frac{ \alpha }{ | \alpha | } \in \{1,-1\}$.
		Then $s K_X$ is ample and the following are equivalent:
		\begin{enumerate}[\rm(i)]
			
			\item $B$ is not rigid;
			
			\item for some $d \geq 1$, $B^{(d)}$ is not rigid;
			
			\item there exists a homogeneous element $h \in B \setminus \{0\}$ of positive degree such that the open subset $D_+(h)$ of $\Proj B$ is a cylinder;
			
			\item there exists a $(sK_X)$-polar cylinder of $X$.
			
		\end{enumerate}
	\end{enumerate}
\end{theorem}

\begin{proof}
		First, since $\cotype(a_0,\dots,a_n) = 0$, $X$ is a well-formed quasismooth weighted hypersurface (by Proposition \ref{PBWellFormed}) so (a) follows from Theorem \ref{dualizing}. 
		
		We prove (b). By Lemma \ref{7fr6d543sqderfvuwR8u390rucyq74etf} (b), if assertion (b) is true for one particular choice of a homogeneous element $T$ in $\Frac B$ of degree 1, then it is true for every choice of such a $T$. As such, we assume henceforth that $T = \prod_{i = 0}^n x_i^{b_i}$ where $b_0, \dots, b_n \in \Integ$ are such that $\sum_{i = 0}^n b_i\deg(x_i) = 1$. By Lemma \ref{7fr6d543sqderfvuwR8u390rucyq74etf} (a), $H^*=\div_\Omega(T) = \sum_{i=0}^n b_i C^\Omega_{\pgoth_i}$. Letting $E = \sum_{i=0}^n b_i C^X_{\pgoth_i} \in \Div(X)$ we obtain that $E^* = \sum_{i=0}^n b_i C^\Omega_{\pgoth_i} = \div_\Omega(T) = H^*$
		(since by Lemma \ref{PBramification}, $e_{\pgoth_i} = 1$ for all $i = 0,\dots,n$). Since $D \mapsto D^*$ is injective, we obtain that $E = H$, so $H \in \Div(X)$. By (a) together with Theorem \ref{Demazure},  $\OSheaf_X(K_X) \isom \omega_X \isom \OSheaf_X(\alpha) \isom \OSheaf_X(\alpha H)$ and so $K_X \sim \alpha H$. Finally, $K_X$ is $\Rat$-Cartier by \ref{cyclicQuotientSingularitiesPBNothing}, proving (b).
		  
        We prove (c). Let $H$ be as in part (b). Since $H$ is ample, $rH$ is ample for every integer $r>0$. Since $K_X \sim \alpha H$, $sK_X \sim |\alpha|H$ is ample. By Remark \ref{PBsatCodim1}, $B$ is saturated in codimension 1. By Theorem \ref{edh83yf6r79hvujhxu6wrefji9e}, (i), (ii) and (iii) are equivalent. Assume (iii) holds, and let $h \in B \setminus\{0\}$ be a homogeneous element of positive degree such that $D_+(h) \subset X$ is a cylinder. By Lemma \ref{HpolarPrincipal}, $D_+(h)$ is $H$-polar and since $sK_X \sim |\alpha|H$ it is also $(sK_X)$-polar by Remark \ref{HH'polar}, so (iii) implies (iv). Conversely, assume (iv) holds. Let $U$ be an $(sK_X)$-polar cylinder of $X$.  Since $sK_X  \sim |\alpha|H$, $U$ is  $H$-polar by Remark \ref{HH'polar}. By Lemma \ref{HpolarPrincipal}, (iii) holds and so (iv) implies (iii), proving (c).    
\end{proof}

\begin{remark}
    We will see in Corollary \ref{Tplusrigid} that when $\alpha > 0$, items (i)-(iv) in Theorem \ref{PBCanonical2} (c) never hold. 
\end{remark}

\subsection{The case of non-negative  amplitude}\label{posAmplitude}
In this section, we show that every ring $B_{a_0,\dots,a_n}$ with $(a_0,\dots,a_n) \in \Gamma_n^+$ is rigid. 

\begin{comment}
\begin{lemma} \label{noCylinder} Let  $X$ be a normal projective surface with quotient singularities such that $K_{X}$ is pseudoeffective. Then $X$ does not contain a cylinder. 
\end{lemma}
\begin{proof} Assume $X$ contains a cylinder $U\cong Z\times\mathbb{A}^{1}$. Let $\sigma:\tilde{X}\to X$ be a resolution of indeterminacy of the induced rational map $\rho:X\dasharrow\bar{Z}$, where $\bar{Z}$ is the smooth projective model of $Z$ such that $\sigma:\tilde{X}\to X$ is also a resolution of singularities of $X$. Let $\tilde{\rho}:\tilde{X}\to\bar{Z}$ be the induced $\mathbb{P}^{1}$-fibration and let $C\subset\tilde{X}$ be the unique curve contained in $\sigma^{-1}(X\setminus U)$ which is a section of $\tilde{\rho}$. Consider the ramification formula $K_{\tilde{X}}=\sigma^{*}K_{X}+\sum_{i=1}^r b_{i}E_{i}$ for $\sigma$, where $\cup_{i = 1}^r E_{i}$ are the exceptional curves of $\sigma$. Since $X$ has quotient and hence klt singularities, $b_{i}>-1$ for all $i  =1, \dots, r$. For a general closed fiber $F\cong \mathbb{P}^1$ of $\tilde{\rho}$, we have $$-2=K_{\tilde{X}}\cdot F=(\sigma^{*}K_{X}+\sum_{i=1}^r b_{i}E_{i})\cdot F=K_{X}\cdot\sigma_{*}F+\sum_{i=1}^r b_{i}(E_{i}\cdot F).$$ Since $K_X$ is pseudoeffective, $K_{X}\cdot\sigma_{*}F\geq 0$. Since $E_{i}\cdot F\geq 0$ for every exceptional curve $E_{i}$, the only possibility is that $C= E_{i}$ for some $i$. Assuming that $C = E_{1}$, we have, since $E_{1}\cdot F=1$, that $ b_1\leq-2$, which is impossible. 
\end{proof}
\end{comment}

\begin{remark}\label{higherAdjunction} We recall that if $Y$ is a nonsingular variety of dimension $n$ and $C$ is a nonsingular curve in $Y$, the adjunction formula gives the following isomorphism of sheaves on $C$:  
$$ \omega_C = \bigwedge^{n-1}\mathcal{N}_{C/Y} \otimes \omega_Y|_C.$$
It follows as a consequence that if $Y$ is a nonsingular variety and $\phi : Y \to B$ is a $\PPP^1$-fibration with general fiber $L$, then $K_Y \cdot L = -2$. 
\end{remark}

\begin{proposition}\label{noCylinder}
    Let $X$ be a normal projective variety with at most log-canonical singularities such that $K_X$ is pseudoeffective and $\Rat$-Cartier. Then $X$ does not contain a cylinder. 
\end{proposition}

\begin{proof}
    Assume $X$ contains a cylinder $U\cong Z\times\mathbb{A}^{1}$ for some affine variety $Z$. Replacing $Z$ by a suitable open smooth subvariety, we may assume without loss of generality that $Z$ is smooth. Let $\bar{Z}$ be a smooth completion of $Z$. The projection $\pr : U \to Z$ induces a rational map $\rho : X \to \bar{Z}$. Let $\sigma:\tilde{X}\to X$ be a birational morphism that resolves the indeterminacy of the induced rational map $\rho:X\dasharrow\bar{Z}$, such that $\sigma$ is also a log resolution of singularities of $X$. Then by construction of $\tilde{X}$, the birational map $\phi = \rho \circ \sigma : \tilde{X} \to \bar{Z}$ is a well-defined proper morphism. By generic smoothness applied to $\phi$ (Corollary 10.7 in \cite{Hartshorne}), there is a non-empty open set $V \subseteq \bar{Z}$ such that $\phi|_{\phi^{-1}(V)} : \phi^{-1}(V) \to V$ is a smooth morphism. Without loss of generality, we may assume $V \subseteq Z$. Since each fiber of $\phi|_{\phi^{-1}(V)}$ is regular of dimension 1 and contains an affine line, it follows that $ \phi^{-1}(V) \isom V \times \PPP^1$ and $\phi|_{\phi^{-1}(V)} : V \times \PPP^1 \to V$ is a trivial $\PPP^1$-bundle extending the cylinder $V \times \aff^1$. Let $U'$ denote this smaller cylinder $V \times \aff^1$. 

    Let $V \times \{\infty\} = \phi^{-1}(V) \setminus U' = (V \times \PPP^1)\setminus (V \times \aff^1)$ and let $\Sigma$ denote the closure of $V \times \{\infty\}$ in $\tilde{X}$; note that $\Sigma$ is a divisor of $\tilde{X}$ contained in $\tilde{X} \setminus \sigma^{-1}(U')$. Consider the ramification formula for the log resolution $\sigma$
    $$K_{\tilde{X}} = \sigma^*K_X + \sum_{i \in I} a_i E_i$$ 
    where the $E_i$ are the exceptional divisors of $\sigma$. For a general closed fiber $L$ of $\phi:\tilde{X} \to \bar{Z}$ we have  
    \begin{equation}\label{ramificationAgain} -2 = K_{\tilde{X}}\cdot L = \sigma^*K_X \cdot L + \sum_{i \in I} a_i E_i \cdot L = K_X \cdot \sigma_*L  + \sum_{i \in I} a_i E_i \cdot L,
    \end{equation}
    the first equality by Remark \ref{higherAdjunction}, the second by observing that the projection formula (see Proposition 2.3(c) of \cite{fulton2012intersection}) still holds for $\Rat$-Cartier divisors. Since $X$ has log-canonical singularities, $a_i \geq -1$ for all $i \in I$ and since $L$ is a general fiber we obtain that $E_i \cdot L \geq 0$ for all $i$ and $E_i \cdot L > 0$ if and only if $E_i = \Sigma$ in which case $\Sigma \cdot L = 1$. Also, since $K_{\tilde{X}}$ is pseudoeffective and $\sigma_*L$ is general and effective, $K_X \cdot \sigma_*L \geq 0$. It follows that the right hand side of \eqref{ramificationAgain} is at most $-1$, a contradiction. We conclude that $X$ cannot contain a cylinder. 
\end{proof}

	\begin{corollary}\label{Tplusrigid}
		Let $(a_0,\dots,a_n) \in \Gamma_n^+$ and let $B = B_{a_0,\dots,a_n}$.
		Then $B^{(d)}$ is rigid for every $d\ge1$.
	\end{corollary}
	
	\begin{proof}
		Since $\cotype(a_0, \dots, a_n) = 0$, Proposition \ref{PBWellFormed} implies that $X = \Proj(B)$ is a well-formed quasismooth weighted complete intersection. By \ref{cyclicQuotientSingularitiesPBNothing}, $X$ is normal and has cyclic quotient singularities. Moreover, Theorem \ref{PBCanonical2} (b) (using the notation of said theorem) implies that $H \in \Div(X)$, $H$ is ample, and $K_X = \alpha H$ where by assumption $\alpha \geq 0$. So, $K_X$ is either trivial or ample and in particular is pseudoeffective. Proposition \ref{noCylinder} then implies that $X$ does not contain a cylinder and so Theorem \ref{edh83yf6r79hvujhxu6wrefji9e} (d) implies $B^{(d)}$ is rigid for all $d \in \Nat^+$.
	\end{proof}

\section{Proof of the Main Conjecture in dimension $3$}\label{Sec:Cotype0}

First observe that by Corollary \ref{reductionCor} together with the $n = 3$ case of Corollary \ref{Tplusrigid}, to prove the $n = 3$ case of the Main Conjecture, it suffices to show that $B_{a_0, a_1, a_2, a_3}$ is rigid for all $(a_0, a_1, a_2, a_3) \in \Gamma_3^-$.

\begin{proposition}\label{antiCanonical}
    Let $(a_0,a_1,a_2,a_3) \in \Gamma_3^{-}$, let $B = B_{a_0,a_1,a_2,a_3}$ and let $X = \Proj B$. Then the following hold:

	\begin{enumerate}[\rm(a)]
		\item $X$ is a del Pezzo surface. 
		
		\item There exists $d \in \Nat^+$ such that $B^{(d)}$ is not rigid if and only if $X$ contains a $-K_X$-polar cylinder.	
	\end{enumerate}  
\end{proposition}
\begin{proof}
	By Theorem \ref{PBCanonical2} (c) and (b), $-K_X$ is an ample $\Rat$-Cartier divisor. Since $X$ has quotient singularities (by \ref{cyclicQuotientSingularitiesPBNothing}), $X$ is a del Pezzo surface, proving (a). Part (b) follows from Theorem \ref{PBCanonical2} (c).  
\end{proof}

\subsection{The simpler cases}

The following further reduces the proof of the $n = 3$ case of the Main Conjecture %\ref{PBConjecture} 
to eight specific cases:

\begin{lemma}$($\cite[Lemma 4.2.4]{ChitayatThesis}$)$\label{whatsInTminus}
	Up to a permutation of $a_0, a_1, a_2, a_3$, the set $\Gamma_3^-$ consists of the following $4$-tuples:
	$$\{(2,3,3,6),\ \  (2,3,6,6),\ \  (2,4,4,4),\ \  (3,3,3,3),\ \  (3,3,4,4),\ \  (3,3,5,5),\ \  (2,3,4,12),\ \  (2,3,5,30)\}.$$
\end{lemma}

\begin{nothing}$($\cite[Sections 4.5-4.6]{ChitayatThesis}$)$\label{easyCases} For each $(a_0,a_1, a_2, a_3) \in \Gamma_3^-$, $\Proj B_{a_0,a_1,a_2,a_3}$ is a del Pezzo surface with quotient singularities. After determining the degrees of these del Pezzo surfaces and their respective singularity types, one can apply Theorem \ref{thm:CPWantiCanonical} to show that for each $(a_0,a_1,a_2,a_3) \in \{(2,3,3,6),\ \  (2,3,6,6),\ \  (2,4,4,4),\ \  (3,3,3,3)\}$, $\Proj B_{a_0,a_1,a_2,a_3}$ does not contain an anti-canonical polar cylinder. For $(a_0,a_1,a_2,a_3) \in  \{(3,3,4,4),\ \  (3,3,5,5)\}$, Lemmas 4.1 and 5.1 in \cite{Cheltsov2010} imply that for every effective anti-canonical $\mathbb{Q}$-divisor $D$ on $X$, the log pair $(X,D)$ is log-canonical. Thus, by Lemma  \ref{cylinderSDNotLogCanonical}, $X$ does not contain an anti-canonical polar cylinder. It then follows from Proposition \ref{antiCanonical} (b), that for each $(a_0,a_1,a_2,a_3) \in \Gamma_3^- \setminus \{(2,3,4,12), \ (2,3,5,30)\}$, $(B_{a_0,a_1,a_2,a_3})^{(d)}$ is rigid for every $d \geq 1$. Combining this analysis with Corollary \ref{Tplusrigid}, we obtain:
\end{nothing}
\begin{corollary}\label{toFinish}
To finish the proof of the Main Conjecture in dimension 3, it suffices to prove that $B_{2,3,5,30}$ and $B_{2,3,4,12}$ are rigid.
\end{corollary}

\subsection{Rigidity of $B_{2,3,5,30}$}\label{23530Rigid}

\begin{nothing}\label{23530Intro}
        Let $B = B_{2,3,5,30}$ and let $S = \Proj B \subset \PPP(15,10,6,1)$. Consider the degree 1 homogeneous element $T = x_3 \in \Frac(B)$. Then $\Delta = V_+(x_3) \in \Div(S)$ is the unique $\Rat$-divisor satisfying $B = \bigoplus_{n \in \Nat} B_n$ where $B_n = \text{\rm H}^0(S,\OSheaf_S(n\Delta)) T^n$ for all $n \in \Nat$ (as defined in Theorem \ref{Demazure}). Considering $\Delta$ as a closed subvariety of $S$, we find $\Delta \isom \PPP^1.$

        Since $\cotype(2,3,5,30) = 0$, Proposition \ref{PBWellFormed} implies that $S$ is a well-formed hypersurface. Since (by \ref{cyclicQuotientSingularitiesPBNothing}) $S$ is normal and Cohen-Macaulay, we have $\omega_S \overset{\ref{NormalCMCoincide}}{\isom} \omega_S^o \overset{\ref{dualizing}}{\isom} \OSheaf_S(-2) \overset{\ref{Demazure}}{\isom} \OSheaf_S(-2\Delta)$. It follows that $2\Delta$ is an ample anti-canonical divisor of $S$ and it can be checked that $(K_S)^2 = \frac{2}{15}$. In particular, $S$ is a singular del Pezzo surface. Since $S$ is well-formed, $\Sing(S) = S \cap \Sing(\PPP(15,10,6,1)) = \{[0:1:-1:0], [1:0:-1:0],[1:-1:0:0]\}$, and it can be checked that 
        \begin{itemize}
    \item $[1:-1:0:0]$ is a $\frac{1}{5}(1,1)$ singularity,
    \item $[1:0:-1:0]$ is a $\frac{1}{3}(1,1)$ singularity,    
    \item $[0:1:-1:0]$ is a $\frac{1}{2}(1,1)$ singularity.
        \end{itemize}
        Moreover, for each $P \in \Sing(S)$, $\mult_P(\Delta) = 1$.
\end{nothing}

\begin{nothing}\label{23530IntroPart2}
    For each $k \in \{2,3,5\}$, let $P_k$ denote the $\frac{1}{k}(1,1)$ singularity of $S$. Let $\sigma : \tilde{S} \to S$ be the minimal resolution of singularities of $S$ and let $\tilde{E}_k$ denote the exceptional curve lying over $P_k$. Example \ref{Cyclicdiscrepancy} shows that for each $k \in \{2,3,5\}$,  $\tilde{E}_k \isom \PPP^1$ is a $(-k)$-curve and
    \begin{equation}\label{canonicalFormula}
    K_{\tilde{S}}=\sigma^{*}K_{S}-\frac{1}{3}\tilde{E}_3-\frac{3}{5}\tilde{E}_5.
    \end{equation}
     Let $\tilde{\Delta}$ denote the proper transform of $\Delta$ on $\tilde{S}$. Since $1 = \mult_{P_2}(\Delta) = \mult_{P_3}(\Delta)= \mult_{P_5}(\Delta)$ we have  $\sigma^{*}\Delta=\tilde{\Delta}+\frac{1}{2}\tilde{E}_2+\frac{1}{3}\tilde{E}_3 + \frac{1}{5}\tilde{E}_5$. Since $2\Delta$ is an anti-canonical divisor of $S$, we obtain using \eqref{canonicalFormula} that 
    \begin{equation}\label{anticanonicalStilde}   2\tilde{\Delta}+\tilde{E}_2+\tilde{E}_3+\tilde{E}_5 \text{ is an anti-canonical divisor of $\tilde{S}$ and $K_{\tilde{S}}^{2} = -2$. }
    \end{equation}
    The support of $\sigma^*(\Delta)$ is given by the following weighted graph

    \[
	\begin{array}{ccccc}
	&  & (\tilde{E}_2,-2)\\
	&  & |\\
	(\tilde{E}_3,-3) & - & (\tilde{\Delta},-1) & - & (\tilde{E}_5,-5).
	\end{array}
	\]
    
\end{nothing}
\subsubsection*{Auxiliary surfaces and birational morphisms}
\begin{nothing}\label{23530Setup}
    We now describe some birational morphisms from $\tilde{S}$. Note that $S, \tilde{S}$ and $\sigma$ are already defined, whereas the other surfaces and morphisms will be defined below.  

    $$
		\xymatrix{
			\tilde{S} \ar[d]_\sigma \ar[r]^{\tau_1} & S' \ar[r]^{\tau_2} & \breve{S} \ar[r]^{\tau_3} & \hat{S}  \\		
			S &&&
		}
    $$

    Define $\tau_1: \tilde{S} \to S'$ to be the contraction of $\tilde{\Delta}$ onto the smooth point of $S'$ which we denote by $x_0'$. For each $i \in \{2,3,5\}$, let $E_i' = {\tau_1}_*(\tilde{E}_i)$. The support of ${\tau_1}_*(\sigma^*\Delta) = \frac{1}{2}E_2' + \frac{1}{3}E_3' + \frac{1}{5}E_5'$ is a union of a $(-1)$-curve, a $(-2)$-curve and a $(-4)$-curve  intersecting at a single point. 

    Define $\tau_2 :  S'  \to \breve{S}$ to be the contraction of ${E_2}'$ onto a smooth point of $\breve{S}$ which we denote by $\breve{x_0}$. Let $\breve{E}_i = {\tau_2}_*(E_i')$ for each $i \in \{3,5\}$. Then $\breve{E_3}^2 = -1$, $\breve{E_5}^2 = -3$ and $\breve{E_3},\breve{E_5}$ are projective lines that intersect tangentially; in particular $\breve{E_3} \cdot \breve{E_5} = 2$. Define $\tau_3 :  \breve{S}  \longrightarrow \hat{S}$ to be the contraction of  $\breve{E_3}$ onto a smooth point of $\hat{S}$ which we denote by $\hat{x_0}$. Let $\hat{E}_5 = {\tau_3}_{*}(\breve{E}_5)$. Since $\breve{E}_3 \cdot \breve{E}_5 = 2$, $\mult_{\hat{x_0}}(\hat{E}_5) = 2$ and $\hat{E}_5^2 = \breve{E}_5^2 +(2)(2) = 1$. It follows that $\hat{E}_5$ is a singular projective curve containing an affine line, hence is a cuspidal curve with a cusp at $\hat{x_0}$. In summary, we have the following diagrams representing the support of $\sigma^*(\Delta) \subset \tilde{S}$ and its image after contracting the $(-1)$-curves described above.

    \begin{equation*}  \label {threepictures}
\setlength{\unitlength}{1mm}
\scalebox{.8}{\fbox{\begin{picture}(35,45)(-20,-3)
\put(0,0){\line(0,1){40}}
\put(-10,10){\line(1,0){20}}
\put(-10,20){\line(1,0){20}}
\put(-10,30){\line(1,0){20}}
\put(-2,0){\makebox(0,0){\tiny $\tilde \Delta$}}
\put(-12,10){\makebox(0,0)[r]{\tiny $\tilde E_{2}$}}
\put(-12,20){\makebox(0,0)[r]{\tiny $\tilde E_{3}$}}
\put(-12,30){\makebox(0,0)[r]{\tiny $\tilde E_{5}$}}
\put(8,11){\makebox(0,0)[b]{\tiny $-2$}}
\put(8,21){\makebox(0,0)[b]{\tiny $-3$}}
\put(8,31){\makebox(0,0)[b]{\tiny $-5$}}
\put(-1,38){\makebox(0,0)[r]{\tiny $-1$}}
\put(-18,-2){\makebox(0,0)[b]{$\tilde S$}}
\end{picture}}
\quad $\xrightarrow{\tau_1}$ \quad
\fbox{\begin{picture}(35,45)(-20,-3)
\put(0,20){\circle*{1}}
\put(-10,20){\line(1,0){20}}
\put(0,20){\line(1,1){10}}
\put(0,20){\line(1,-1){10}}
\put(0,20){\line(-1,1){10}}
\put(0,20){\line(-1,-1){10}}
\put(-12,10){\makebox(0,0)[r]{\tiny $E_{2}'$}}
\put(-12,20){\makebox(0,0)[r]{\tiny $E_{3}'$}}
\put(-12,30){\makebox(0,0)[r]{\tiny $E_{5}'$}}
\put(8,9){\makebox(0,0)[t]{\tiny $-4$}}
\put(8,21){\makebox(0,0)[b]{\tiny $-2$}}
\put(8,31){\makebox(0,0)[b]{\tiny $-1$}}
\put(0,18){\makebox(0,0)[t]{\tiny $x_0'$}}
\put(-18,-2){\makebox(0,0)[b]{$S'$}}
\end{picture}}
\quad $\xrightarrow{\tau_2}$ \quad
\fbox{\begin{picture}(35,45)(-17,-3)
\qbezier(-10,30)(10,20)(-10,10)
\qbezier(10,30)(-10,20)(10,10)
\put(0,20){\circle*{1}}
\put(-12,10){\makebox(0,0)[r]{\tiny $\breve E_{3}$}}
\put(12,10){\makebox(0,0)[l]{\tiny $\breve E_{5}$}}
\put(-10,28){\makebox(0,0)[r]{\tiny $-1$}}
\put(9,28){\makebox(0,0)[l]{\tiny $-3$}}
\put(-15,-2){\makebox(0,0)[b]{$\breve S$}}
\put(1,20){\makebox(0,0)[l]{\tiny $\breve x_0$}}
\end{picture}}
\quad $\xrightarrow{\tau_3}$ \quad
\fbox{\begin{picture}(35,45)(-17,-3)
\qbezier(10,30)(8,25)(1,20)
\qbezier(10,10)(8,15)(1,20)
\put(0,20){\circle*{1}}
\put(12,10){\makebox(0,0)[l]{\tiny $\hat E_{5}$}}
\put(10,28){\makebox(0,0)[l]{\tiny $1$}}
\put(-15,-2){\makebox(0,0)[b]{$\hat S$}}
\put(-4,20){\makebox(0,0)[l]{\tiny $\hat x_0$}}
\end{picture}}}
\end{equation*}
    
\end{nothing}

\begin{proposition}\label{prop:ShatdelPezzo}
		With the notation of \ref{23530Setup}, 
		\begin{enumerate}[\rm(a)] 
			\item $\hat{E}_5$ is an ample anti-canonical divisor of $\hat{S}$; 
 			\item $\hat{S}$ is a smooth del Pezzo surface of degree 1;	    
		\end{enumerate}
	\end{proposition}
	\begin{proof}
		We prove (a). It follows from \eqref{anticanonicalStilde} and from the definitions of $\tau_1, \tau_2, \tau_3$ that $\hat{E}_5$ is an anti-canonical divisor of $\hat{S}$. By Theorem \ref{NakaiMoishezon}, it suffices to show that for any curve $\hat{C} \subset \hat{S}$, $\hat{C} \cdot \hat{E}_5 > 0$. Observe that the birational map $S \dashrightarrow \hat{S}$ restricts to an isomorphism between $S \setminus \Delta$ and $\hat{S}\setminus \hat{E}_5$. If $\hat{C} = \hat{E}_5$, we are done since $\hat{E}_5^2 = 1$.  If $\hat{C} \neq \hat{E}_5$, the proper transform $\tilde{C} \in \tilde{S}$ of $\hat{C} \in \hat{S}$ is equal to the proper transform of some irreducible curve $C$ in $S$ other than $\Delta$.  Since $\Delta$ is ample, $C \cdot \Delta >0$ and so $\tilde{C}$ intersects $\tilde{\Delta} \cup \tilde{E}_2 \cup \tilde{E}_3\cup \tilde{E}_5$. This implies that $\hat{C}$ intersects $\hat{E}_5$ and so $\hat{C} \cdot \hat{E}_5 > 0$. This proves (a). 
  
  For (b), $\hat{S}$ is smooth because $\tau_1, \tau_2, \tau_3$ contract $(-1)$-curves on smooth surfaces. The other claims in (b) follow from part (a) and the fact that $\hat{E}_5^2 = 1$. 
  \end{proof}

\subsubsection{Exclusion of anti-canonical polar cylinders in $\Proj B_{2,3,5,30}$.}\label{subsub23530}

This section shows that $S = \Proj B_{2,3,5,30}$ does not contain an anti-canonical polar cylinder.

\begin{lemma}\label{lcinDelta}
Let $D$ be an effective anti-canonical $\Rat$-divisor on $S$ such that the log pair $(S,D)$ is not log-canonical at
a point $p \in S$. Then
\begin{enumerate}[\rm(a)]
\item $p \in \Supp(\Delta)$
\item $\Delta \subseteq \Supp(D)$.
\end{enumerate}
\end{lemma}
\begin{proof}
    For (a), assume that $p \notin \Supp(\Delta)$. Then, by the discussion in \ref{23530Intro}, $p$ is a regular point of $S$ and so by Lemma \ref{lem:multlc} we have $\mult_p(D) > 1$. Consider the complete linear system $\Meul = |-5K_S|$ on $S$. Since $-5K_S \sim 10\Delta$, the elements of $|-5K_S|$ have form $V_+(f)$ where $f = ax_1 + bx_2x_3^4 + cx_3^{10}$ is homogeneous of degree 10 and $[a:b:c] \in \PPP^2$. Let $\Meul_p$ denote the subsystem of $\Meul$ consisting of elements passing through $p$. Since the condition that a member of $\Meul$ passes through $p$ imposes one linear condition, the subsystem $\Meul_p$ is one dimensional. Moreover, a general member of $\Meul_p$ is irreducible and $\Meul_p$ has no fixed components. Consequently, there exists an irreducible $M \in \Meul_p$ such that $\Supp(M) \nsubseteq \Supp(D)$. Since $\mult_p(M) \geq 1$ (recalling from \ref{23530Intro} that $K_S^2 = \frac{2}{15}$), we have 
    $$
    \frac{2}{3} = 5(-K_S)^2 = M \cdot D \geq \mult_p(M) \cdot \mult_p(D) > 1 
    $$
    which is impossible. We conclude that $p \in \Supp(\Delta)$, proving (a).

    We prove (b). Suppose $\Delta$ is not an irreducible element of  $\Supp(D)$. If $p$ is a regular point of $S$, again Lemma \ref{lem:multlc} implies $\mult_p(D) > 1$. Since $p \in \Supp(\Delta)$, we obtain 
    $$\frac{1}{15} = D \cdot \Delta \geq \mult_p(D) \cdot \mult_p(\Delta) > 1$$
    which is impossible, so $p$ is a singular point of $S$. We have $p = P_k$ (as in \ref{23530IntroPart2}) for some $k \in \{2,3,5\}$ and so $p$ is of type $\frac{1}{k}(1,1)$. Let $\mu : \bar{S} \to S$ denote the minimal resolution of the point $p$. The exceptional divisor $\bar{E}$ is a projective line and $\bar{E}^2 = -k$. We have
    \begin{enumerate}[\rm(i)]
        \item $K_{\bar{S}} = \mu^*K_S - \frac{k-2}{k}\bar{E}$,
        \item $\mu^*\Delta = \bar{\Delta} + \frac{1}{k} \bar{E}$ where $\bar{\Delta}$ is the strict transform of $\Delta$,  
        \item $\mu^*D = \bar{D} + \frac{1}{k}\mult_p(D)\bar{E}$, 
    \end{enumerate}
    where claim (i) follows from Example \ref{Cyclicdiscrepancy}, and (ii) and (iii) are because $p$ is of type $\frac{1}{k}(1,1)$. It follows that 
    $$\mu^*(K_S + D) = K_{\bar{S}} + \bar{D} + \frac{1}{k}(\mult_p(D) + k-2)\bar{E}.$$ 
    Since $(S,D)$ is not log-canonical at $p$, there exists a point $\bar{p} \in \bar{E}$ such that the log pair $(\bar{S}, \bar{D} + \frac{1}{k}(\mult_p(D) + k-2)\bar{E})$ is not log-canonical at $\bar{p}$. Since $\bar{p}$ is a regular point of $\bar{S}$, Lemma \ref{lem:multlc} implies
    $$1 < \mult_{\bar{p}}(\bar{D} + \frac{1}{k}(\mult_p(D) + k-2)\bar{E}) = \mult_{\bar{p}}(\bar{D}) + \frac{1}{k}(\mult_p(D) + k-2).$$ 
    Since $\mult_{\bar{p}}(\bar{D}) \leq \mult_p(D)$, we find $\frac{k+1}{k} \mult_p(D) > \frac{2}{k}$ and hence that $\mult_p(D) > \frac{2}{k+1}$.  Since $\Delta$ is not an irreducible component of $\Supp(D)$, we have
    $$\frac{1}{15} = \Delta \cdot D \geq (\Delta \cdot D)_p \geq \frac{1}{k} \mult_p(D) > \frac{2}{k(k+1)} \geq \frac{1}{15}$$ since $k \in \{2,3,5\}$. This is clearly impossible, and proves (b). 
    
\end{proof}

\begin{proposition}\label{Main23530}
    The surface $S = \Proj B_{2,3,5,30}$ does not contain an anti-canonical polar cylinder. 
\end{proposition}
\begin{proof}
    Suppose $D$ is an effective anti-canonical $\Rat$-divisor such that $S \setminus \Supp(D)$ a cylinder $U \isom \Aff^1 \times Z$ and let $\rho : S \dashrightarrow \PPP^1$ be the rational map induced by the projection $\mathrm{pr}_{Z} : U \to Z$. By the discussion in \ref{cylinderSetup}, one of the following holds: 
    \begin{enumerate} [\rm(i)]
        \item $\rho$ is a $\PPP^1$-fibration and $D$ contains an irreducible component, say $C_1$, which is a section of $\rho$;
        \item $\rho$ has a unique proper base point $p$ on $S$.
    \end{enumerate}
In case (i), statement \eqref{intermediate} in the proof of Lemma \ref{cylinderSDNotLogCanonical} implies that the log pair $(S,D)$ is not log-canonical at a point of $C_1$. In case (ii), by Lemma \ref{cylinderSDNotLogCanonical}, the log pair $(S,D)$ is not log-canonical at the base point $p$. In both cases (i) and (ii), Lemma \ref{lcinDelta} (b) implies that $\Delta \subseteq \Supp(D)$. Consider the $\Rat$-divisor $\tilde{D} = \sigma^*(D) + \frac{1}{3}\tilde{E}_3 + \frac{3}{5}\tilde{E}_5$. Since $P_2, P_3, P_5 \in \Supp(\Delta)$, and since $\Delta \subseteq \Supp(D)$,  $\tilde{D}$ contains $\tilde{E}_2, \tilde{E}_3,\tilde{E}_5$ and $\tilde{\Delta}$ in its support. Also, $\tilde{D}$ is an anti-canonical $\Rat$-divisor of $\tilde{S}$ by \eqref{canonicalFormula}. Recalling the notation of \ref{23530Setup}, let $\tau = \tau_3 \circ \tau_2 \circ \tau_1$. The contraction $\tau : \tilde{S} \to \hat{S}$ restricts to an isomorphism between $\sigma^{-1}(U) = \tilde{S} \setminus \Supp(\tilde{D})$ and its image $\tau(\sigma^{-1}(U)) = \hat{S} \setminus \Supp({\hat{D}})$ where $\hat{D} = \tau_*(\tilde{D})$. Since $\hat{D}$ is an effective anti-canonical $\Rat$-divisor on $\hat{S}$, it follows that $\hat{S}$ contains an anti-canonical polar cylinder. Since (by Proposition \ref{prop:ShatdelPezzo}) $\hat{S}$ is a smooth del Pezzo surface of degree 1, this contradicts Theorem \ref{thm:CPWantiCanonical}. 
\end{proof}
\subsection{Rigidity of $B_{2,3,4,12}$.}\label{23412Rigid}

\begin{nothing}\label{23412setup}
        Let $B = B_{2,3,4,12}$ and let $S = \Proj B \subset \PPP(6,4,3,1)$. Consider the degree 1 homogeneous element $T = x_3 \in \Frac(B)$. Then $\Delta = V_+(x_3) \in \Div(S)$ is the unique $\Rat$-divisor satisfying $B = \bigoplus_{n \in \Nat} B_n$ where $B_n = \text{\rm H}^0(S,\OSheaf_S(n\Delta)) T^n$ for all $n \in \Nat$ (as defined in Theorem \ref{Demazure}). Considering $\Delta$ as a closed subvariety of $S$, we find $\Delta \isom \PPP^1.$

        Since $\cotype(2,3,4,12) = 0$, Proposition \ref{PBWellFormed} implies that $S$ is a well-formed hypersurface. Since (by \ref{cyclicQuotientSingularitiesPBNothing}) $S$ is normal and Cohen-Macaulay, we have $\omega_S \overset{\ref{NormalCMCoincide}}{\isom} \omega_S^o \overset{\ref{dualizing}}{\isom} \OSheaf_S(-2) \overset{\ref{Demazure}}{\isom} \OSheaf_S(-2\Delta)$. It follows that $2\Delta$ is an ample anti-canonical divisor of $S$ and it can be checked that $(K_S)^2 = \frac{2}{3}$. In particular, $S$ is a singular del Pezzo surface. Since $S$ is well-formed, $\Sing(S) = S \cap \Sing(\PPP(6,4,3,1)) = \{[1:-1:0:0], [1:0:\zeta_8:0], [1:0:\zeta_8^{-1}:0]\}$ where $\zeta_8$ is a primitive $8^{th}$ root of unity, and it can be checked that 
        \begin{itemize}
    \item $[1:-1:0:0]$ is a $\frac{1}{2}(1,1)$ singularity,
    \item $[1:0:\zeta_8:0]$ and $[1:0:\zeta_8^{-1}:0]$ are $\frac{1}{3}(1,1)$ singularities.
        \end{itemize}
        Moreover, for each $P \in \Sing(S)$, $\mult_P(\Delta) = 1$.
\end{nothing}
\begin{nothing}\label{sigmaDef}
    Let $P_2$ denote the $\frac{1}{2}(1,1)$ singularity of $S$ and let $P_{3+}$ and $P_{3-}$ denote the $\frac{1}{3}(1,1)$ singularities. Let $\sigma : \tilde{S} \to S$ be the minimal resolution of singularities of $S$ and let $\tilde{E}_2$, $\tilde{E}_{3+}$ and $\tilde{E}_{3-}$ denote the exceptional curves lying over $P_{2}, P_{3+}$ and $P_{3-}$ respectively. Example \ref{Cyclicdiscrepancy} shows that for each $k \in \{2,3+,3-\}$,  $\tilde{E}_k \isom \PPP^1$, $\tilde{E}_2^2 = -2$, $\tilde{E}_{3+}^2 = \tilde{E}_{3-}^2 = -3$ and
    \begin{equation}\label{canonicalFormula2}
    K_{\tilde{S}} \sim \sigma^{*}K_{S}-\frac{1}{3}\tilde{E}_{3+}-\frac{1}{3}\tilde{E}_{3-}.
    \end{equation}

    Let $\tilde{\Delta}$ denote the proper transform of $\Delta$ on $\tilde{S}$. Since $1 = \mult_{P_2}(\Delta) = \mult_{P_{3+}}(\Delta)= \mult_{P_{3-}}(\Delta)$ we have  $\sigma^{*}\Delta=\tilde{\Delta}+\frac{1}{2}\tilde{E}_2+\frac{1}{3}\tilde{E}_{3+} + \frac{1}{3}\tilde{E}_{3-}$. Since $2\Delta$ is an anti-canonical divisor of $S$, we obtain using \eqref{canonicalFormula2} that 
    \begin{equation}\label{anticanonicalStilde2}   2\tilde{\Delta}+\tilde{E}_2+\tilde{E}_{3+}+\tilde{E}_{3-} \text{ is an anti-canonical divisor of $\tilde{S}$ and $K_{\tilde{S}}^{2} = 0$. }
    \end{equation}
    The support of $\sigma^*(\Delta)$ is given by the following weighted graph

    \[
	\begin{array}{ccccc}
	&  & (\tilde{E}_2,-2)\\
	&  & |\\
	(\tilde{E}_{3+},-3) & - & (\tilde{\Delta},-1) & - & (\tilde{E}_{3-},-3).
	\end{array}
	\]

\end{nothing} 

\subsubsection*{Auxiliary surfaces and birational morphisms}
\begin{nothing}\label{23412Setup}
    We describe some birational morphisms from $\tilde{S}$. Note that $S, \tilde{S}$ and $\sigma$ are already defined, whereas the other surfaces and morphisms will be defined below.  

$$
		\xymatrix{
			\tilde{S} \ar[d]_\sigma \ar[r]^{\tau_1} & \hat{S} \ar[r]^{\tau_2} & \breve{S}  \\		
			S &&
		}
    $$
    
\end{nothing}

   Let $\tau_1: \tilde{S} \to \hat{S}$ be the contraction of $\tilde{\Delta}$ onto a smooth point $\hat{x_0}$ of $\hat{S}$. For each $i \in \{2,3+,3-\}$, let $\hat{E}_i = {\tau_1}_*(\tilde{E}_i)$. The support of ${\tau_1}_*(\sigma^*\Delta) = \hat{E}_2 + \hat{E}_{3+} + \hat{E}_{3-}$ is a union of a $(-1)$-curve, and two $(-2)$-curves all intersecting at $\hat{x_0}$. 
	
    Let $\tau_2:\hat{S} \to \breve{S}$ denote the contraction of $\hat{E}_2$ onto a smooth point $\breve{x_0}$ of $\breve{S}$ and for each $i \in \{3+,3-\}$, let $\breve{E}_i = {\tau_2}_*(\hat{E}_i)$. Then $\breve{E}_{3+}$ and $\breve{E}_{3-}$ are $(-1)$-curves that intersect tangentially at $\breve{x_0}$. In summary, we have the following diagrams representing the support of $\sigma^*(\Delta) \subset \tilde{S}$ and its image after contracting the $(-1)$-curves described above.

    \begin{equation*}  \label {threepictures}
\setlength{\unitlength}{1mm}
\scalebox{.8}{\fbox{\begin{picture}(35,45)(-20,-3)
\put(0,0){\line(0,1){40}}
\put(-10,10){\line(1,0){20}}
\put(-10,20){\line(1,0){20}}
\put(-10,30){\line(1,0){20}}
\put(-2,0){\makebox(0,0){\tiny $\tilde \Delta$}}
\put(-12,10){\makebox(0,0)[r]{\tiny $\tilde E_{2}$}}
\put(-12,20){\makebox(0,0)[r]{\tiny $\tilde E_{3+}$}}
\put(-12,30){\makebox(0,0)[r]{\tiny $\tilde E_{3-}$}}
\put(8,11){\makebox(0,0)[b]{\tiny $-2$}}
\put(8,21){\makebox(0,0)[b]{\tiny $-3$}}
\put(8,31){\makebox(0,0)[b]{\tiny $-3$}}
\put(-1,38){\makebox(0,0)[r]{\tiny $-1$}}
\put(-18,-2){\makebox(0,0)[b]{$\tilde S$}}
\end{picture}}
\quad $\xrightarrow{\tau_1}$ \quad
\fbox{\begin{picture}(35,45)(-20,-3)
\put(0,20){\circle*{1}}
\put(-10,20){\line(1,0){20}}
\put(0,20){\line(1,1){10}}
\put(0,20){\line(1,-1){10}}
\put(0,20){\line(-1,1){10}}
\put(0,20){\line(-1,-1){10}}
\put(-12,10){\makebox(0,0)[r]{\tiny $\hat{E}_{2}$}}
\put(-12,20){\makebox(0,0)[r]{\tiny $\hat{E}_{3+}$}}
\put(-12,30){\makebox(0,0)[r]{\tiny $\hat{E}_{3-}$}}
\put(8,9){\makebox(0,0)[t]{\tiny $-2$}}
\put(8,21){\makebox(0,0)[b]{\tiny $-2$}}
\put(8,31){\makebox(0,0)[b]{\tiny $-1$}}
\put(0,18){\makebox(0,0)[t]{\tiny $\hat{x}_0$}}
\put(-18,-2){\makebox(0,0)[b]{$\hat{S}$}}
\end{picture}}
\quad $\xrightarrow{\tau_2}$ \quad
\fbox{\begin{picture}(35,45)(-17,-3)
\qbezier(-10,30)(10,20)(-10,10)
\qbezier(10,30)(-10,20)(10,10)
\put(0,20){\circle*{1}}
\put(-12,10){\makebox(0,0)[r]{\tiny $\breve E_{3+}$}}
\put(12,10){\makebox(0,0)[l]{\tiny $\breve E_{3-}$}}
\put(-10,28){\makebox(0,0)[r]{\tiny $-1$}}
\put(9,28){\makebox(0,0)[l]{\tiny $-1$}}
\put(-15,-2){\makebox(0,0)[b]{$\breve S$}}
\put(1,20){\makebox(0,0)[l]{\tiny $\breve x_0$}}
\end{picture}}}
\end{equation*}

    \begin{lemma}\label{SDotDelPezzo}
        With notation as in \ref{23412Setup}, $\breve{S}$ is a smooth del Pezzo surface of degree 2. 
    \end{lemma}
    \begin{proof}
        Observe that $\breve{E}_{3+} + \breve{E}_{3-}$ is an effective anti-canonical divisor on $\breve{E}$ satisfying $(\breve{E}_{3+} + \breve{E}_{3-})^2 = 2$. To prove $\breve{E}_{3+} + \breve{E}_{3-}$ is ample, by Theorem \ref{NakaiMoishezon} it suffices to show that $\breve{C}\cdot (\breve{E}_{3+} + \breve{E}_{3-}) > 0$ for every irreducible curve $\breve{C}$ in $\breve{S}$. If $\breve{C} \in \{\breve{E}_{3+}, \breve{E}_{3-}\}$, then $\breve{C}\cdot (\breve{E}_{3+} + \breve{E}_{3-}) = 1$. Otherwise, since $\Delta$ is ample, $\breve{C}$ is the image of some curve intersecting $\sigma^{-1}(\Delta)$ and so $\breve{C}\cdot (\breve{E}_{3+} + \breve{E}_{3-}) > 0$.        
    \end{proof}

\subsubsection*{Exclusion of anti-canonical polar cylinders in $\Proj B_{2,3,4,12}$.}

This subsection shows that $S = \Proj B_{2,3,4,12}$ does not contain an anti-canonical polar cylinder.

\begin{lemma}\label{linearsystem3Delta} With the notation of \ref{23412setup},

\begin{enumerate}[\rm(a)]
    \item a member of the complete linear system $|3 \Delta|\subset \Div(S)$ other than $3 \Delta$ is an integral curve $C$ on $S$ of form $V_+(f)$ where $f = x_2 + \lambda x_3^3$ for some $\lambda \in \Comp$. 
    \item The intersection of $C \cap (S \setminus \Delta) \isom C \cap D_+(x_3)$ is isomorphic to $\Spec(\Comp[X_0, X_1] / \lb X_0^2 + X_1^3 + \lambda^4 + 1 \rb$ which is nonsingular when $\lambda^4 \neq -1$ and is isomorphic to the cuspidal plane curve $V(X_0^2 + X_1^3) \subset \Aff^2$ when $\lambda^4 = -1$. 
    \item For every member $C \in |3 \Delta|$ other than $3\Delta$, the log pair $(S, \frac{2}{3} C)$ is log-canonical at every point of $S \setminus \Delta$.
\end{enumerate}

\end{lemma}
\begin{proof}
    Parts (a) and (b) are easy to verify. For (c), we note that since $S \setminus \Delta$ is regular, if $(S, \frac{2}{3} C)$ is not log-canonical at a point $p$, then $p$ must be the unique singular point of $C$. Computing a log resolution of $(S,\frac{2}{3}C)$, we find three exceptional curves lying over $p$, whose discrepancies are $-\frac{1}{3}, 0$ and $-\frac{1}{3}$.    
\end{proof}

We state Lemma 2.2 in \cite{cheltsov_park_won_2016}. We note that while the authors assume that their surface $S$ has at most Du Val singularities, both the statement and proof remain valid for normal projective surfaces with quotient singularities.

\begin{lemma}\label{epsilonChanges}$($\cite[Lemma 2.2]{cheltsov_park_won_2016} $)$
    Let $S$ be a normal projective surface with at most quotient singularities. Let $D = \sum_{i = 1}^r a_i C_i$ where $C_i$ is a prime divisor and $a_i > 0$ for all $i$ and let $T = \sum_{i = 1}^r b_i C_i$ where $b_i \geq 0$ for all $i$. Suppose furthermore that
    \begin{itemize}
    \item $D \sim_\Rat T$, 
    \item $D$ and $T$ are distinct.
    \end{itemize}
    For every non-negative rational $\epsilon$, set $D_\epsilon = (1 + \epsilon)D - \epsilon T$. Then

    \begin{enumerate}[\rm(a)]
        \item $D_\epsilon \sim_\Rat D$ for every $\epsilon \geq 0$;
        \item the set $\setspec{\epsilon \in \Rat^+}{D_\epsilon \text{ is effective}}$ attains its supremum, which we will denote by $\mu$;
        \item there exists an irreducible component of $\Supp(T)$ that is not contained in $\Supp(D_\mu)$;
        \item if the log pair $(S,T)$ is log-canonical at a point $p$ but $(S,D)$ is not log-canonical at $p$, then $(S, D_\mu)$ is not log-canonical at $p$.
    \end{enumerate}
\end{lemma}

\begin{lemma}\cite[Lemma 3.4]{affineCones} \label{CPWdegree2}
		Let $S$ be a smooth del Pezzo surface of degree $2$ and let $D$ be an effective $\Rat$-divisor such that $D \sim_\Rat -K_S$. Suppose that the log pair $(S,D)$ is not log-canonical at $p$. Then 
  \begin{enumerate}[\rm(a)]
      \item there exists a unique divisor $C  \in |-K_S|$ such that $(S,C)$ is not log-canonical at $p$.
      \item The support of $D$ contains all the irreducible components of $C$ where either
      \begin{itemize}
      \item $C$ is an irreducible rational curve with a cusp at $p$
      \item $C = C_1 + C_2$ where $C_1$ and $C_2$ are $(-1)$-curves meeting tangentially at $p$.
      \end{itemize}
  \end{enumerate}
  \end{lemma}

\begin{lemma}\label{whereIsp}
Let $D$ be an effective anti-canonical $\Rat$-divisor on $S$ such that the log pair $(S,D)$ is not log-canonical at
a point $p$. Then
\begin{enumerate}[\rm(a)]
    \item $p \in \Delta$,
    \item $\Delta \subseteq \Supp(D)$.
\end{enumerate}
\end{lemma}
\begin{proof}
    Assume $p \in S \setminus \Delta$. Since $p$ is a regular point of $S$, Lemma \ref{lem:multlc} implies that $\mult_p(D) > 1$. Observe that there exists a unique member $C_p \in |3\Delta|$ passing through $p$. If $C_p$ is not contained $\Supp(D)$, then (recalling from \ref{23412setup} that $K_S^2 = \frac{2}{3}$)
    $$1 = \frac{3}{2}(K_S)^2 = C_p \cdot D \geq \mult_p(C_p) \mult_p(D) > 1$$
    which is impossible. It follows that 
    \begin{equation}\label{supportContainsCp}
        \text{$C_p$ is contained in $\Supp(D)$.}
    \end{equation}

    Let $T = \frac{2}{3}C_p$ and note that $T \sim_\Rat 2\Delta \sim_\Rat D$. Since $(S,T)$ is log canonical at $p$ (by Lemma \ref{linearsystem3Delta}(c)) but $(S,D)$ is not, we have in particular $D \neq T$. By Lemma \ref{epsilonChanges}, if we define $D_\epsilon = (1+\epsilon)D-\epsilon T$ for $\epsilon \in \Rat^+$, then the maximum element $\mu$ of
    $\setspec{ \epsilon \in \Rat^+}{D_\epsilon \geq 0}$ exists and $D_\mu \in \Div(S,\Rat)$ satisfies:
    $$
    D_\mu \geq 0, \qquad D_\mu \sim_\Rat -K_S, \qquad (S, D_\mu) \text{ is not log canonical at $p$}, \qquad C_p \nsubseteq \Supp(D_\mu).
    $$
    On the other hand, applying \eqref{supportContainsCp} to $D_\mu$ shows that $C_p \subseteq \Supp(D_\mu)$. This contradiction proves (a).
    
We prove (b). If $\Delta \subseteq \Supp(D)$, the proof is complete, so suppose $\Delta$ is not an irreducible component of $\Supp(D)$. If $p$ is a regular point of $S$, then $\mult_p(D) > 1$ by Lemma \ref{lem:multlc} and by (a) we find 
$$\frac{1}{3} = D \cdot \Delta \geq \mult_p(D) \cdot \mult_p(\Delta) > 1$$
which is impossible. Thus, $p \in \{P_2,P_{3+},P_{3-}\}$ is a singular point of $S$ of type $\frac{1}{k}(1,1)$ where $k \in \{2,3\}$. Let $\mu : \bar{S} \to S$ denote the minimal resolution of the point $p$. The exceptional divisor $\bar{E}$ is a projective line and $\bar{E}^2 = -k$. We have
    \begin{enumerate}[\rm(i)]
        \item $K_{\bar{S}} = \mu^*K_S - \frac{k-2}{k}\bar{E}$,
        \item $\mu^*\Delta = \bar{\Delta} + \frac{1}{k} \bar{E}$ where $\bar{\Delta}$ is the strict transform of $\Delta$,  
        \item $\mu^*D = \bar{D} + \frac{1}{k}\mult_p(D)\bar{E}$.  
    \end{enumerate}
    
    Since $(S,D)$ is not log-canonical at $p$, there exists a point $\bar{p} \in \bar{E}$ such that the log pair $(\bar{S}, \bar{D} + \frac{1}{k}(\mult_p(D) + k-2)\bar{E}$ is not log-canonical at $\bar{p}$. Since $\bar{p}$ is a regular point of $\bar{S}$ and of $\bar{E}$, it follows that 
    $$1 < \mult_{\bar{p}}(\bar{D} + \frac{1}{k}(\mult_p(D) + k-2)\bar{E}) = \mult_{\bar{p}}(\bar{D}) + \frac{1}{k}(\mult_p(D) + k-2) \leq \mult_p D + \frac{1}{k}(\mult_p(D) + k-2)$$ 
    since $\mult_{\bar{p}}(\bar{D}) \leq \mult_p(D)$. We find $\frac{k+1}{k} \mult_p(D) > \frac{2}{k}$ and hence that $\mult_p(D) > \frac{2}{k+1}$.  Since $\Delta$ is not an irreducible component of $\Supp(D)$, we have
    $$\frac{1}{3} = \Delta \cdot D \geq (\Delta \cdot D)_p \geq \frac{1}{k} \mult_p(D) > \frac{2}{k(k+1)}$$ which implies that $k > 2$. Thus $k = 3$ and so $p \in \{P_{3+},P_{3-}\}$.
    
    Without loss of generality, we may assume $p = P_{3+}$. With the notation of \ref{sigmaDef}, we have $-K_{\tilde{S}} \sim \sigma^*(K_S) + \frac{1}{3}\tilde{E}_{3+} + \frac{1}{3}\tilde{E}_{3-} \sim 2\tilde{\Delta} + \tilde{E}_2 + \tilde{E}_{3+} + \tilde{E}_{3-}$ where $\tilde{\Delta}$ is the strict transform of $\Delta$ and $\tilde{\Delta}^2 = -1$. The divisor $\tilde{D} = \sigma_*(D)+ \frac{1}{3}\tilde{E}_{3+} + \frac{1}{3}E_{3-}$ is an effective anti-canonical $\Rat$-divisor on $\tilde{S}$ and $\tilde{\Delta} \cdot \tilde{D} = \tilde{\Delta} \cdot(2\tilde{\Delta} + \tilde{E}_2 + \tilde{E}_{3+} + \tilde{E}_{3-}) = 1$. Since the log pair $(S,D)$ is not log-canonical at $P_{3+}$, the log pair $(\tilde{S}, \tilde{D})$ is not log-canonical at some point $\tilde{p} \in \tilde{E}_{3+}$. Since $\Delta \nsubseteq \Supp(D)$, $\tilde{\Delta} \nsubseteq \Supp(\tilde{D})$. Also, since $\tilde{S}$ is regular at every point of $\tilde{E}_{3+}$, we must have $\tilde{p} \in \tilde{E}_{3+} \setminus \bar{\Delta}$ (otherwise we would have inequality $1 = (\tilde{\Delta} \cdot \tilde{D}) \geq (\tilde{\Delta} \cdot \tilde{D})_{\tilde{p}} \geq \mult_{\tilde{p}} \tilde{D} > 1$). Let $\tau = \tau_2 \circ \tau_1: \tilde{S} \to \breve{S}$ denote the contraction of $\tilde{\Delta} \cup \tilde{E}_2$ onto the smooth point $\breve{x_0}$ of the smooth degree 2 del Pezzo surface $\breve{S}$ (by Lemma \ref{SDotDelPezzo}). Then (by \ref{23412Setup}) $\breve{E}_{3+} + \breve{E}_{3-}$ is anti-canonical integral divisor on $\breve{S}$ consisting of two $(-1)$-curves intersecting tangentially at $\breve{x_0} \in \breve{S}$ and $\breve{D} = \tau_* \tilde{D}$ is an effective anti-canonical $\Rat$-divisor on $\breve{S}$ containing $\breve{E}_{3+}$ and $\breve{E}_{3-}$ in its support. Since the log pair $(\tilde{S}, \tilde{D})$ is not log-canonical at the point $\tilde{p} \in \tilde{E}_{3+} \setminus \tilde{\Delta}$, the log pair $(\breve{S}, \breve{D})$ is not log-canonical at the point $\tau(\tilde{p}) \neq \breve{x_0}$. Let $\breve{p} = \tau(\tilde{p})$. By Lemma \ref{CPWdegree2} (a), there exists a unique anti-canonical divisor $\breve{C}$ on $\breve{S}$ such that the log pair $(\breve{S}, \breve{C})$ is not log-canonical at $\breve{p}$ (in which case $\mult_{\breve{p}}(\breve{C}) > 1$). By Lemma \ref{CPWdegree2} (b), $\Supp(\breve{C}) \subseteq \Supp(\breve{D})$. Since $\breve{E}_{3+}$ is a $(-1)$-curve on $\breve{S}$ and $\breve{C}$ is an anti-canonical divisor on $\breve{S}$, the adjunction formula implies $\breve{E}_{3+} \cdot \breve{C} = 1$. This implies that $\breve{C}$ contains $\breve{E}_{3+}$ in its support (otherwise we would obtain the impossible $1 = \breve{E}_{3+} \cdot \breve{C} \geq (\breve{E}_{3+} \cdot \breve{C})_p \geq \mult_{\breve{p}}(\breve{C}) > 1$). Since $\breve{E}_{3+}$ does not have a cusp, Lemma \ref{CPWdegree2} (b) gives that $\breve{C} = \breve{E}_{3+} + \breve{C'}$ where $\breve{C'}$ is a $(-1)$-curve intersecting $\breve{E}_{3+}$ tangentially at $\breve{p}$. Since $\breve{p} \neq \breve{x_0}$, we have $\breve{C'} \neq \breve{E}_{3-}$ and since $\breve{E}_{3-}$ is a $(-1)$-curve, we obtain $1 = \breve{C} \cdot \breve{E}_{3-} = (\breve{E}_{3+} + \breve{C'})  \cdot \breve{E}_{3-} = 2 + \breve{C'} \cdot \breve{E}_{3-} > 2$. This contradiction shows our initial assumption $\Delta \nsubseteq \Supp(D)$ is impossible and completes the proof of (b). 
     
    \end{proof}

    \begin{proposition}\label{Main23412}
        The surface $S = \Proj(B_{2,3,4,12})$ does not contain an anti-canonical polar cylinder.
    \end{proposition}

    \begin{proof}
        Suppose $S$ contains an anti-canonical polar cylinder $U = S \setminus \Supp(D)$. Let $\tilde{D} = \sigma^*(D) + \frac{1}{3}\tilde{E}_{3+} +\frac{1}{3}E_{3-}$ and observe that by \eqref{canonicalFormula2}, $\tilde{D}$ is an anticanonical $\Rat$-divisor of $\tilde{S}$. Since $U$ is smooth, $P_2, P_{3+}, P_{3-} \in \Supp(D)$ and so $\tilde{E}_2, \tilde{E}_{3+}, \tilde{E}_{3-}$ are contained in $\Supp(\tilde{D})$. By Lemma \ref{cylinderSDNotLogCanonical}, there exists a point $p \in S$ such that the log pair $(S,D)$ is not log-canonical at $p$. By Lemma \ref{whereIsp}, $\Delta \subseteq \Supp(D)$ and so $\tilde{\Delta}$ is also contained in $\Supp(\tilde{D})$. Now $\sigma^{-1}(U) = \tilde{S} \setminus \Supp(\sigma^*(D)) = \tilde{S} \setminus \Supp(\tilde{D})$ is an anti-canonical polar cylinder in $\tilde{S}$. Let $\tau = \tau_2 \circ \tau_1 : \tilde{S} \to \breve{S}$ be as in \ref{23412Setup}. Then $\breve{S} \setminus \Supp(\tau_*\tilde{D}) = \tau(\sigma^{-1}(U))$ is an anti-canonical polar cylinder in the smooth degree 2 del Pezzo surface $\breve{S}$ (by Lemma \ref{SDotDelPezzo}). This contradicts Theorem \ref{thm:CPWantiCanonical}.  
    \end{proof}

\begin{corollary}\label{hardCases}
    Let $B = B_{a_0,a_1,a_2,a_3}$ where $(a_0,a_1,a_2,a_3) \in \{(2,3,5,30), (2,3,4,12)\}$. Then $B^{(d)}$ is rigid for all $d\in \Nat^+$.
\end{corollary}

\begin{proof}
    Let $X = \Proj B$. By Propositions \ref{Main23530} and \ref{Main23412}, $X$ does not contain an anti-canonical polar cylinder, so by Theorem \ref{antiCanonical} (b),  $B^{(d)}$ is rigid for all $d \in \Nat^+$.
\end{proof}

We can finally prove the Main Theorem. 

\begin{maintheorem}
    Let $n \geq 2$, and let $B_{a_0, a_1,a_2,a_3} = \bk[X_0, X_1,X_2,X_3] / \langle X_0^{a_0} + X_1^{a_1} + X_2^{a_2} +  X_3^{a_3} \rangle$ be a Pham-Brieskorn ring. If $\min\{a_0, a_1, a_2,a_3 \} \geq 2$ and at most one element $i$ of $\{0,1,2,3\}$ satisfies $a_i = 2$, then $B_{a_0,a_1,a_2,a_3}$ is rigid. Moreover, if $\cotype(a_0, a_1, a_2, a_3) = 0$, then $(B_{a_0, a_1, a_2, a_3})^{(d)}$ is rigid for all $d \in \Nat^+$. 
\end{maintheorem}
\begin{proof}
The claim that $B_{a_0. a_1, a_2, a_3}$ is rigid follows immediately from Corollary \ref{toFinish} and Corollary \ref{hardCases}. The ``moreover" follows from Corollary \ref{BdRigid}.   
\end{proof}

\section{Remarks in Higher Dimensions}\label{Sec:Higher}

\begin{nothing}
	We conclude by offering some remarks on the Main Conjecture %\ref{PBConjecture} 
 in higher dimensions (where $n > 3$). Recall that by Theorem \ref{reduction}, it suffices to prove the conjecture for the cases $B_{a_0, \dots, a_n}$ where $\cotype(a_0, \dots, a_n) = 0$, namely, the cases where $\Proj B_{a_0, \dots, a_n}$ is a well-formed quasismooth weighted hypersurface. Note that the canonical sheaf and canonical divisor of $X = \Proj B_{a_0, \dots, a_n}$ can still be computed using Theorems \ref{dualizing} and \ref{PBCanonical2}.	
\end{nothing}

\begin{lemma}
	Let $n \geq 4$ and let $X = \Proj B_{a_0, \dots, a_n}$. Then $\CaCl(X) \isom \Integ$. 
\end{lemma}
\begin{proof}

    For every $(b_0, \dots, b_n)$, there exists some $(a_0, \dots, a_n)$ such that $\Proj(B_{b_0, \dots, b_n}) \isom \Proj(B_{a_0, \dots, a_n})$ where  $\cotype(a_0, \dots, a_n) = 0$. (See Lemma 3.1.14 (b) in \cite{ChitayatThesis}.) So, we may assume that $\cotype(a_0, \dots, a_n) = 0$. Since $X$ is integral, $\CaCl(X) \isom \Pic(X) \isom \Integ$ (the second isomorphism by Theorem 3.2.4 of \cite{dolgachev} since $\dim(X) \geq 3$).
\end{proof}

\begin{proposition}\label{allHpolar}
	Let $X$ be a projective normal $\Rat$-factorial variety and suppose that $\CaCl(X) \isom \Integ$. Let $H$ be any ample $\Rat$-divisor on $X$. Then every cylinder $U$ in $X$ is an $H$-polar cylinder. 
\end{proposition}

\begin{proof}
	Let $U$ be a cylinder in $X$. Since $U$ us affine, the inclusion $i:U \to X$ is an affine morphism. Corollary 21.12.7 of \cite{EGA4} implies that $U = X \setminus \Supp(D)$ for some effective divisor $D \in \Div(X)$. Since $X$ is $\Rat$-factorial, we may assume that $D$ is Cartier. Let $H$ be any ample $\Rat$-divisor of $X$. We must show that $U$ is $H$-polar.

	If $D = 0$, then $\Supp(D) = \emptyset$ and so $X = U$ is both affine and projective. This implies that $X$ is a point, contradicting the assumption that $\CaCl(X) \isom \Integ$; it follows that $D \neq 0$. Since $D$ is a nonzero effective divisor and $\Pic(X) \isom \CaCl(X) \isom \Integ$, Example 1.2.4 of \cite{lazarsfeld2004positivity1} implies that $D$ is ample.
	
	Choose $s \in \Nat^+$ such that $sH$ is Cartier. Then both $sH$ and $D$ belong to $\CaDiv(X)$. Since $\CaCl(X) \isom \Integ$, there exist $m,n \in \Integ$ such that $m > 0$ and $m(sH) \sim nD$. Since $m(sH)$ is ample, so is $nD$, and it can be checked that $n > 0$. Since $m(sH) \sim nD$ where $m, s, n \in \Nat^+$, Remark \ref{HH'polar} implies that every $D$-polar cylinder is $H$-polar. Since $U$ is $D$-polar, it is therefore $H$-polar.
\end{proof}

\begin{corollary}\label{notRigidIffCylinder}
	Let $n \geq 4$, and let $B = B_{a_0, \dots, a_n}$ be such that $\cotype(a_0, \dots, a_n)= 0$. Then $B$ is not rigid if and only if $\Proj B$ contains a cylinder. 
\end{corollary}
\begin{proof}
	Let $\alpha$ denote the amplitude of $X = \Proj B$. If $\alpha \geq 0$ then $K_X$ is pseudoeffective, $X$ does not contain a cylinder by Proposition \ref{noCylinder} and $B$ is rigid by Corollary \ref{Tplusrigid}. Otherwise $\alpha < 0$ and  $B$ is not rigid if and only if $X$ contains a $-K_X$-polar cylinder if and only if $X$ contains a cylinder. (The first equivalence is by Theorem \ref{PBCanonical2} (c), the second is by Proposition \ref{allHpolar}.)
\end{proof}

\begin{remark}
It is well known that for $n \geq 4$, $B_{a_0, \dots, a_n}$ is a unique factorization domain. (See Proposition 4.9.7 of \cite{ChitayatThesis} for one proof.) For $n = 2$ and $n = 3$ however, Exercise II.6.5 in \cite{Hartshorne} and \cite{storch1984picard} show that this is generally not the case.
\end{remark}

\begin{nothing}
	Let $n \geq 4$, suppose $\cotype(a_0, \dots, a_n) = 0$ and let $X = \Proj(B_{a_0, \dots, a_n})$. As in the $n = 3$ case, the problem naturally splits into the cases where $-K_X$ is ample or $K_X$ is pseudoeffective.
		
	If $\alpha < 0$,  then Theorem \ref{PBCanonical2} implies that $-K_X$ is an ample $\Rat$-Cartier divisor, so $X$ is a Fano variety and by Corollary \ref{notRigidIffCylinder}, we are motivated to investigate the more general question of existence of cylinders in singular Fano varieties, a question which is generally open in dimension 2 or larger. 
	
    If $\alpha \geq 0$, then $X$ has quotient (hence log-canonical) singularities and $K_X$ is either an ample or trivial $\Rat$-Cartier divisor and hence is pseudoeffective. It follows from Proposition \ref{noCylinder} that $X$ cannot contain a cylinder and from Corollary \ref{notRigidIffCylinder} that $B_{a_0, \dots, a_n}$ is rigid. 
\end{nothing}    

\bibliographystyle{alpha}
\bibliography{remainingbibliography}
\end{document}